\documentclass[english,reqno]{amsart}
\usepackage[margin=1in]{geometry}

\usepackage{relsize}

\usepackage{preamble}

\newcommand{\rd}{\mathrm{d}}

\begin{document}

\title{Applications of kinetic tools to inverse transport problems}

\author{Qin Li} 
\address{Mathematics Department and Wisconsin Institutes of Discoveries, University of Wisconsin-Madison, 480 Lincoln Dr., Madison, WI 53705 USA.}
\email{qinli@math.wisc.edu}
\author{Weiran Sun}
\address{Department of Mathematics, Simon Fraser University, 8888 University Dr., Burnaby, BC V5A 1S6, Canada}
\email{weiran\_sun@sfu.ca}

\date{\today}

\thanks{The research of Q.L.~was supported in part by National Science Foundation under award 1619778, 1750488, KI-net 1107291, and Wisconsin Data Science Initiative. The research of W.S.~was supported by NSERC Discovery Grant R611626. Both authors would like to thank Laurent Desvillettes for the helpful discussions.}

\begin{abstract}
We show that the inverse problems for a class of kinetic equations can be solved by classical tools in PDE analysis including energy estimates and the celebrated averaging lemma. Using these tools, we give a unified framework for the reconstruction of the absorption coefficient for transport equations in the subcritical and critical regimes. Moreover, we apply this framework to obtain, to the best of our knowledge, the first result in a nonlinear setting. We also extend the result of recovering the scattering coefficient in~\cite{CS98} from 3D to 2D convex domains. 
\end{abstract}
%
\maketitle
\section{Introduction}

Kinetic theory 
describes the behavior of a large number of particles that follow the same physical laws in a statistical manner. Depending on the particular type of particles, various equations are derived. These include, among many others, the Boltzmann equation for the rarified gas, Vlasov-Poisson equation for charged plasma particles, the radiative transfer equation for photons, and the neutron transport equation for neutrons. In the kinetic theory, one uses $f(t, x, v)$ to denote the density distribution function of the particles in the phase space $(x,v)$ at time $t$. The kinetic equation that $f$ satisfies is of the form
\[
\partial_t f + v\cdot\nabla_xf +E\cdot\nabla_v f = Q[f]\,,
\]
where the terms on the left characterizes the trajectory of particles moving with velocity $v$ and accelerated/decelerated by the external field $E$, and the term on the right collects information about particles colliding with each other and/or with the media. The specific form of $Q$ depends on the particular type of particles studied.

During the past three decades, analysis of kinetic equations has seen drastic progresses. In particular, with the introduction of averaging lemma and application of the concept of entropy combined with traditional energy estimates, the well-posedness and the convergence to equilibria can now be shown for many kinetic equations.

Despite their wide applications for forward problems, such techniques are barely used in the inverse setting, where the goal is to recover certain unknown parameters (in $E$ or $Q$ for example). These parameters are usually set constitutively or ``extracted" from lab experiments. Mathematically, such ``extraction" is a process termed inverse problem, which is generally hard to solve rigorously. Aside from very limited examples~\cite{CS2,CS3,SU2d,CS98,time_harmonic,BalTamasan,BalMonard_time_harmonic,SU2008,LLU18,StefanovU,MamonovRen,SUlens} along with some analysis on stability~\cite{Wang1999,Bal14,Bal10,ILW16,Ldiff,NUW,Yamamoto2016,LLU18,CLW,ZhaoZ18}, it is unknown in general, what kind of data would be enough to guarantee a unique reconstruction or when the reconstruction is stable. What is more, in the few solved examples, the techniques used highly depend on careful and explicit calculations of the solutions to the PDEs. As a consequence, it is challenging to extend these results to general models (see reviews in~\cite{Bal_review,Stefanov_2003}). There are, however, a large amount of studies addressing the related computational issues~\cite{Gamba,McCor,LN1983} (also see reviews in~\cite{Arridge99,Kuireview,Arridge_Schotland09}).

In this paper we propose to use energy methods and the averaging lemma to investigate the unique reconstruction of parameters in transport equations in a rather general setup. Since our methods do not rely on fine details of the equation as much as in the previous works, we can apply our results to a class of models including a nonlinear transport equation. We are also able to extend the study of the radiative transport equation in the subcritical case in \cite{CS98,SU2d} to a unified analysis in both subcritical and critical regimes. Further comments regarding the dimensionality can be found in Section~\ref{subsec:Main-Results} where precise statements of the main results are shown.

\subsection{Singular decomposition}
Throughout the paper we study the time-independent problem
\begin{align} \label{eq:RTE-general}
   v \cdot \nabla_x f(x,v) = - \sigma_a f(x,v) + F_f(x) \,,\qquad x \in \Omega \subseteq \R^2 \,,v\in \Ss^1\,,
\end{align}
where $\Omega$ is a bounded convex domain, $\Ss^1$ is the unit circle with a normalized measure, and $F_f(x)$ is a functional of $f$ which only depends on $x$. We assume that $\sigma_a$ is isotropic in the sense that $\sigma_a = \sigma_a(x)$. One  example is the radiative transfer equation (RTE) where $F_f$ is simply defined by taking the zeroth moment of $f$:
\[
F_f(x) = \sigma_s(x) \int_{\Ss^1} f\rd{v}\,.
\]
The data we will be using is of the Albedo type, namely, we can impose an incoming boundary condition and measure the associated outgoing boundary data and define the Albedo operator as
\[
\mathcal{A}:\quad f|_{\Gamma_-}\to f|_{\Gamma_+}\,.
\]
Here $\Gamma_\pm$ are the collections of all coordinates on the physical boundary with the velocity pointing either in or out of the domain defined by
\[
\Gamma_\pm = \{(x,v):\quad x\in\partial\Omega\,, \,\, \pm n(x)\cdot v\leq 0\}\,,
\]
where $n(x)$ is the outward normal at $x\in\partial\Omega$. The goal is  to reconstruct parameters in~\eqref{eq:RTE-general} such as $\sigma_a$ or unknown parameters in $F_f$ by taking multiple sets of incoming-to-outgoing data.

The basic approach we adopt here is the method of singular decomposition. It is introduced in \cite{CS98} to recover the absorption and scattering kernel in the radiative transfer equation. The main idea of this method is built upon the observation that the solution $f$ to~\eqref{eq:RTE-general} can be decomposed into parts with different regularity. Each part contains information of different terms in equation~\eqref{eq:RTE-general}. Hence if one is able to separate these parts with different regularity by imposing proper test functions on $\Gamma_-$, then there is hope to recover various terms in equation~\eqref{eq:RTE-general}.

As an illustration, we explain the basic procedures to reconstruct $\sigma_a$ in~\eqref{eq:RTE-general}. We start with splitting the solution as $f = f_1 + f_2$ where $f_1, f_2$ satisfy
\[
\begin{cases}
v\cdot\nabla_x f_1 = -\sigma_a f_1 \,,\\
f_1|_{\Gamma_-} = f|_{\Gamma_-}
\end{cases} \qquad\text{and}\qquad
\begin{cases}
v\cdot\nabla_x f_2 = -\sigma_a f_2 +F_f \,,\\
f_2|_{\Gamma_-} = 0 \,.
\end{cases}
\]
With a relatively singular and concentrated input, e.g. $f|_{\Gamma_-}=\phi(\frac{x-x_0}{\epsilon})\phi(\frac{v-v_0}{\epsilon})$, $f_1$ will be more singular compared with $f_2$: the information of $f_1$ propagates only in a narrow neighborhood of a ray while $f_2$ is more spread out. Hence one is able to isolate $f_1$ from $f_2$ by measuring the outgoing data only in a small neighborhood of the exit point for $f_1$. It is then clear from the equation for $f_1$ that the absorption coefficient $\sigma_a$ can be fully recovered once $f_1$ known. The details of such analysis is shown in Section~\ref{sec:RTE}. 

The method of singular decomposition has been extensively used in many variations of RTE, including the time-dependent model, when data is angular-averaging type, models with internal source, and models with adjustable frequencies, among some others~\cite{time_harmonic,BalTamasan,BalMonard_time_harmonic,SU2008,LLU18,StefanovU,MamonovRen,SU2d}. See also reviews~\cite{Arridge99,Kuireview,Bal_review}. Stability was discussed in~\cite{Wang1999,Bal14,Bal10,ILW16,Ldiff,CLW,ZhaoZ18}. To our knowledge, all these discussions are centered around linear RTEs. Since linearity plays the central role, so far there has been no result in a nonlinear setup. One of our goals in this paper is to extend singular decomposition to a nonlinear system.

\subsection{Main Results} \label{subsec:Main-Results}
We show two main results in this paper. The first result gives a general framework for recovering the absorption coefficient. To present our idea in the simplest form, we set our proof in two dimension. General dimensions can be similarly treated.
\begin{thm} \label{thm:RTE-general}
Let $\Omega \subseteq \R^2$ be a strictly convex and bounded domain with a $C^1$ boundary. Suppose $\sigma_a \geq 0$ is isotropic and $\sigma_a \in C(\bar\Omega)$. Suppose there exists $p \geq 1$ such that for any given incoming data $\phi$ satisfying 
\begin{align*}
   \phi \in L^p(\Gamma_-) \,, 
\qquad
   \phi \geq 0 \,,
\end{align*} 
equation~\eqref{eq:RTE-general} has a unique nonnegative solution with the bound
\begin{align} \label{bound:RTE-general}
   \norm{F_f}_{L^p(\Omega)} 
\leq 
  C_0 \norm{\phi}_{L^p(\Gamma_-)} \,,
\end{align}
where $C_0$ is independent of $\phi$ and $f$. Then with proper choices of the incoming data and outgoing measurements, the absorption coefficient $\sigma_a$ can be uniquely reconstructed.
\end{thm}
We remark that the assumptions on $F_f$ are not as restrictive as they may appear. In fact it is common for a vast class of kinetic equations that $F_f$ only depends on the moments of $f$ and satisfies the bound in~\eqref{bound:RTE-general}. Upon proving Theorem~\ref{thm:RTE-general} in Section~\ref{sec:RTE}, we will give two examples to demonstrate its effectiveness. 

In the second result, we show the unique recovery of the scattering coefficient $\sigma_s$ in the classical RTE:
\begin{align} \label{eq:RTE-sigma-s}
      v \cdot \nabla_x f = - \sigma_a f + \sigma_s \vint{f} \,,\qquad x \in \Omega \subseteq \R^2 \,,v\in \Ss^1\,,
\end{align}
where $\vint{f} = \int_{\Ss^1} f \dv$ with $\dv$ normalized. 
\begin{thm} \label{thm:scattering}
Let $\Omega \subseteq \R^2$ be a strictly convex and bounded domain with a $C^1$ boundary. Suppose $\sigma_a, \sigma_s \in C(\bar\Omega)$ with $\sigma_a$ given and $0 < \sigma_0 \leq \sigma_s \leq \sigma_a$. Then with proper choices of the incoming data and outgoing measurements, the scattering coefficient $\sigma_s$ in~\eqref{eq:RTE-sigma-s} can be uniquely reconstructed from the measurement of the outgoing ~data.  
\end{thm}

Two comments are in place for Theorem~\ref{thm:scattering}: first, we only show the result in $\R^2$ since this is the case not covered in~\cite{CS98}. Similar strategy used to prove Theorem~\ref{thm:scattering} can also be applied to any higher dimension by using the same incoming data and measurement as in~\cite{CS98}. In this sense, our result is an extension of~\cite{CS98}. Second, in $\R^2$ so far we can only treat the case where $\sigma_s$ is isotropic, that is, $\sigma_s = \sigma_s(x)$. Similar as in~\cite{CS98}, such constraint is not needed for higher dimensions. We also note that 2D case was studied in~\cite{SU2d}. However, there smallness of the scattering kernel is assumed while we can deal with the critical and general subcritical cases. 

This paper is laid out as follows. In Section~\ref{sec:RTE}, we show the proof of Theorem~\ref{thm:RTE-general} together with its applications to the classical linear RTE and a nonlinear RTE coupled with a temperature equation. In Section~\ref{sec:scattering}, we show the proof of Theorem~\ref{thm:scattering}. Some technical parts in the proofs of these two theorems are left in the appendices. 

\section{Absorption Coefficient for Radiative Transfer Equations} \label{sec:RTE}

The domain $\Omega$ considered in this paper is strictly convex with a $C^1$ boundary. More precisely, we assume that there exists a function $\xi: \R^2 \to \R$ such that $\Omega$ and its boundary are described by
\begin{align} \label{def:xi}
  \Omega
= \{x \big| \, \xi(x) \leq 0\}
\qquad\text{and}\qquad
  \del\Omega
= \{x \big| \, \xi(x) = 0\} \,.
\end{align}
We assume that $\nabla_x \xi(x) \neq 0$ for any $x \in \del\Omega$ and there exists a constant $D_0 > 0$ such that
\begin{align} \label{property:xi}
   \sum_{i, j=1}^2 \del_{ij} \xi(x) a_i a_j 
\geq 
   D_0 |a|^2 \,,
\qquad
  \forall \, x \in \Omega \,.
\end{align}
The outward normal $n(x)$ at $x \in \Omega$ is then given by 
\begin{align*}
    n(x) = \frac{\nabla_x \xi(x)}{|\nabla_x \xi(x)|} \,,
\qquad
   \forall \, x \in \del\Omega \,.  
\end{align*}
For each $(x, v) \in \bar\Omega \times \Ss^1$, we use $\tau_-(x, v)$ and $\tau_+(x, v)$ to denote the nonnegative backward and forward exit times, which are the instances where
\begin{align} \label{def:tau-pm}
    x - \tau_-(x, v) v \in \del \Omega \,,
\qquad
    x + \tau_+(x, v) v \in \del \Omega \,,
\quad
  \text{for any $(x, v) \in \bar \Omega \times \Ss^1$.}
\end{align}
Recall the basic properties of the backward exit time from Lemma 2 in \cite{Guo2010}:
\begin{lem}[\cite{Guo2010}] \label{lem:tau-continuity}
Suppose $\Omega \subseteq \R^2$ is strictly convex and has a $C^1$ boundary. Suppose $\xi$ is the characterizing function of $\Omega$ and $\del\Omega$ which satisfies~\eqref{def:xi} and~\eqref{property:xi}. For any $(x, v) \in \bar\Omega \times \Ss^1$, let $\tau_-$ be the backward exit time defined in~\eqref{def:tau-pm} and $x_- \in \del\Omega$ be the exit point given by $x_-= x - \tau_-(x, v) v$. Then
\begin{enumerate}[label=(\alph*), leftmargin=*]
\item $(\tau_-, x_-)$ are uniquely determined for each $(x, v) \in \bar\Omega \times \Ss^1$;

\item Suppose $\xi \in C^1(\R^3)$ and $v \cdot n(x_-) \neq 0$. Then $\tau_-, x_-$ are differentiable at $(x, v)$ with
\begin{align*}
   &\nabla_x \tau_-(x, v) = \frac{n(x_-)}{v \cdot n(x_-)} \,,
\hspace{1.5cm}
   \nabla_v \tau_-(x, v) = \frac{\tau_- n(x_-)}{v \cdot n(x_-)} \,,
\\
  & \nabla_x x_-(x, v) = \Id - \nabla_x \tau_- \otimes v \,,
\qquad
  \nabla_v x_-(x, v) = -\tau_- \Id - \nabla_v \tau_- \otimes v \,. 
\end{align*}
\end{enumerate}
\end{lem}

The rest of this section is devoted to the proof of Theorem~\ref{thm:RTE-general}. As introduced in the previous section, the idea of the proof is to separate the terms in the equations and compare the induced singularities. In particular, let $f$ the solution to the equation~\eqref{eq:RTE-general} with boundary condition $f|_{\Gamma_-} = \phi(x,v)$. We separate it as $f = f_1+ f_2$ so that $f_1$ satisfies
\begin{equation}\label{eqn:f1}
   v \cdot \nabla_x f_1 = - \sigma_a f_1 \,,
\qquad
  f_1 \big|_{\Gamma_-} = \phi(x, v) \,,
\end{equation}
and $f_2$ satisfies
\begin{equation}\label{eqn:f2}
   v \cdot \nabla_x f_2 = - \sigma_a f_2 + F_f(x) \,,
\qquad
  f_2 \big|_{\Gamma_-} = 0 \,.
\end{equation}
If we choose $\phi(x,v)$ to be a delta-like function concentrating at a point $(\Xin, \Vin) \in \Gamma_-$, then it is clear through equation~\eqref{eqn:f1} that most of the information of $f$ will be propagating along the ray 
\[
x=\Xin+\tau\Vin\,,\quad\tau\in[0,\tau_+]\,.
\]
Defining
\begin{equation}\label{eqn:def_out}
   \Vout = \Vin \,, \qquad \Xout = \Xin + \tau_+(\Xin, \Vin) \Vin \,,
\end{equation}
and letting the test function $\psi$ concentrate on $(\Xout,\Vout)$, we will split the measurement of the outgoing data into two components (with $\dGamma = n(x)\cdot v\rd S_x\rd{v}$):
\begin{align}\label{eqn:measurement}
   M_\psi(f)&= \iint_{\Gamma_+} \psi(x, v) f(x, v) \dGamma \nonumber\\
   &= \iint_{\Gamma_+} \psi(x, v) f_1(x, v) \dGamma + \iint_{\Gamma_+}       \psi(x, v) f_2(x, v) \dGamma
\\
   &=  M_\psi(f_1)+M_\psi(f_2)\nonumber \,.
\end{align}
The estimates in the proof are designated to show that
\begin{equation}\label{eqn:goal_f1f2}
M_\psi(f_1)\sim \text{X-ray transform of }\sigma_a\,,\quad M_\psi(f_2)\sim 0\,,
\end{equation}
from which one can reconstruct $\sigma_a$ via the unique recovery of $\sigma_a$ in the X-ray transform. Details of the proof are shown below. One convention that we follow in the rest of this paper is that we repeatedly use $c_0$ and $C_0$ to denote constants that may change from line to line.
\begin{proof}[Proof of Theorem~\ref{thm:RTE-general}]
Let $\Eps, \delta > 0$ be arbitrary constants to be chosen later and let $\phi_0$ be a smooth function on $\R$ such that 
\begin{align*}
  0 \leq \phi_0(r) \leq 1 \,,
\qquad
  \phi_0 \in C^\infty_c([0, \infty)),  
\qquad 
  \phi_0(0) = 1\,,
\qquad
  \int_0^\infty \phi_0(r) \dr = 1 \,.
\end{align*}
For any $(\Xin, \Vin) \in \Gamma_-$ such that
\begin{align} \label{cond:non-degen}
    \Vin \cdot n(\Xin) = -c^{in} < 0 \,,
\end{align}
choose the incoming data for equation~\eqref{eq:RTE-general} as
\begin{align*}
    \phi(x, v) 
= \frac{1}{\Eps \delta}
   \phi_0\vpran{\frac{|x - \Xin|}{\Eps}}
   \phi_0\vpran{\frac{|v - \Vin|}{\delta}} \,,
\qquad
   (x, v) \in \Gamma_- \,.
\end{align*}
Let $(\Xout,\Vout)$ be defined in~\eqref{eqn:def_out}, and we take the test function for measurement to be:
\begin{align*}
   \psi(x, v)
= \psi_0 (x - \Xout) 
   \psi_0 \vpran{\frac{v - \Vout}{\delta}}
= \psi_0 (x - \Xout) 
   \psi_0 \vpran{\frac{v - \Vin}{\delta}} \,,
\qquad
   (x, v) \in \Gamma_+ \,,
\end{align*}
where $\psi_0(r)$ is a smooth function that satisfies 
\begin{align} \label{def:psi-0}
  0 \leq \psi_0(r) \leq 1 \,, \qquad   \psi_0 \in C^\infty_c([0, \infty))\,,  \qquad    \psi_0(0) = 1\,, \qquad   \int_0^\infty \psi_0(r) \dr = 1 \,.
\end{align}
We can solve along characteristics in~\eqref{eqn:f1} and~\eqref{eqn:f2} to obtain explicit and semi-explicit formulas for $f_1$ and $f_2$ as
\begin{align}\label{eqn:f1_soln}
   f_1(x, v)
& = e^{-\int_0^{\tau_-(x, v)} \sigma_a(x - sv) \ds}
      \phi(x - \tau_-(x, v) v, v) \,,
\qquad
    \forall \, (x, v) \in \bar\Omega \times \Ss^1
\end{align}
and
\begin{align}\label{eqn:f2_soln}
   f_2(x,  v)
 = \int_0^{\tau_-(x, v)} e^{-\int_0^s \sigma_a(x - \tau v) \dtau}
     F_f(x - sv)\ds \,,
\qquad
   \forall \, (x, v) \in \bar\Omega \times \Ss^1 \,.
\end{align}
For future use, define the sets $\Ss^1_{x, +}$ and $\del\Omega_v^+$ by
\begin{align}
  \del\Omega_v^+= \left\{x \in \del\Omega \big| n(x) \cdot v > 0\right\} \,,
\qquad
  \text{for all $v \in \Ss^1$} \,, \label{def:Omega-v-plus}
\\
  \Ss^1_{x, +} = \left\{v \in \Ss^1 \big| \, v \cdot n(x) > 0 \right\} \,,
\qquad
  \text{for all $x \in \del\Omega$} \,. \label{def:Ss-x-plus}
\end{align}
We show~\eqref{eqn:goal_f1f2} in two steps. \smallskip

\noindent
\underline{\textbf{Step 1: limit of $M_\psi(f_1)$}} \,
Using~\eqref{eqn:f1_soln}, we have
\begin{align}
   M_\psi(f_1)&= \iint_{\Gamma_+} \psi(x, v)  e^{-\int_0^{\tau_-(x, v)} \sigma_a(x - sv) \ds}\phi(x - \tau_-(x, v) v, v) \dGamma\\
& = \frac{1}{\Eps \delta} \int_{\del\Omega} \int_{\Ss^1_{x, +}} e^{-\int_0^{\tau_-(x, v)} \sigma_a(x - sv) \ds} \psi_0 (x - \Xout) \psi_0 \vpran{\frac{\abs{v - \Vin}}{\delta}} \nn\\
& \hspace{2cm}
     \times  \phi_0 \vpran{\frac{\abs{x - \tau_-(x, v) v - \Xin}}{\Eps}}
      \phi_0 \vpran{\frac{\abs{v - \Vin}}{\delta}}
     n(x) \cdot v \dv\dS_x  \nn
\\
& = \frac{1}{\Eps}
      \int_{\del\Omega}   \psi_0 (x - \Xout)  \,
     G_{\Eps, \delta}(x) \dS_x \,, \label{integral:f-1-RTE-general}
\end{align}
where $G_{\Eps, \delta}(x)$ denotes the inner integral and it can be further simplified in notation as
\begin{align*}
   G_{\Eps, \delta}(x)
&= \frac{1}{\delta} \int_{\Ss^1_{x, +}}
         e^{-\int_0^{\tau_-(x, v)} \sigma_a(x - sv) \ds}
   \psi_0 \vpran{\frac{\abs{v - \Vin}}{\delta}}
\\
& \hspace{2.4cm}
     \times  \phi_0 \vpran{\frac{\abs{x - \tau_-(x, v) v - \Xin}}{\Eps}}
      \phi_0 \vpran{\frac{\abs{v - \Vin}}{\delta}}
     n(x) \cdot v \dv \\
     &=\frac{1}{\delta} \int_{\Ss^1_{x, +}} H_\Eps(x, v)
   \psi_0 \vpran{\frac{\abs{v - \Vin}}{\delta}}
   \phi_0 \vpran{\frac{\abs{v - \Vin}}{\delta}} \dv
\end{align*}
with
\begin{align} \label{def:H-eps}
   H_\Eps(x, v)
&= e^{-\int_0^{\tau_-(x, v)} \sigma_a(x - sv) \ds}
    \phi_0 \vpran{\frac{\abs{x - \tau_-(x, v) v - \Xin}}{\Eps}}
    n(x) \cdot v \,.
\end{align}
We will first pass $\delta \to 0$ and then $\Eps \to 0$ in~\eqref{integral:f-1-RTE-general}. Note that for each fixed $\Eps$ and $x \in \del\Omega$, the inner integral $G_{\Eps, \delta}$ satisfies
\begin{align*}
   0 
\leq   
  G_{\Eps, \delta}(x) 
&\leq
    \frac{1}{\delta} \int_{\Ss^1_{x, +}}
            \psi_0 \vpran{\frac{\abs{v - \Vin}}{\delta}} \dv
\leq \frac{1}{\delta} \int_{0}^{2\pi}
            \psi_0 \vpran{\frac{\sin\theta/2}{\delta/2}} \dv
\\
&\leq
  \frac{1}{\delta} \int_{0}^{\alpha_0 \delta}
            \psi_0 \vpran{\frac{\sin\theta/2}{\delta/2}} \dv
 + \frac{1}{\delta} \int_{2\pi - \alpha_0 \delta}^{2\pi}
            \psi_0 \vpran{\frac{\sin\theta/2}{\delta/2}} \dv
\leq
   2 \alpha_0 \,,
\end{align*}
where $\alpha_0$ only depends on the size of the support of $\psi_0$. Such uniform bound ensures that the Lebesgue Dominated Convergence Theorem can be applied when taking the $\delta$-limit in~\eqref{integral:f-1-RTE-general}. To compute the pointwise $\delta$-limit of $G_{\Eps, \delta}$, we denote $D_{\phi_0}$ as the set where
\begin{align*}
   D_{\phi_0} 
= \left\{(x, v) \in \Gamma_+ \Big| \, 
    \tfrac{\abs{v - \Vin}}{\delta}, \tfrac{|x_- - \Xin|}{\Eps} \in \Supp \phi_0\right\} \,,
\qquad
   x_- = x - \tau_-(x, v) v \,.
\end{align*}
Then the measurement $M_\psi(f_1)$ becomes
\begin{align*}
   M_\psi(f_1)= \iint_{D_{\phi_0}} \psi(x, v) 
         e^{-\int_0^{\tau_-(x, v)} \sigma_a(x - sv) \ds}
      \phi(x - \tau_-(x, v) v, v) \dGamma \,.
\end{align*}
By the non-degeneracy condition of $(\Xin, \Vin)$ in~\eqref{cond:non-degen} and the support of $\phi_0$, the normal direction $n(\cdot)$ is continuous in a small neighbourhood of $\Xin$. Hence, if we choose $\delta, \Eps$ to be small enough, then for any $(x, v) \in D_{\phi_0}$, we have
\begin{align} \label{cond:non-degen-1}
   v \cdot n(x_-) < - \frac{1}{2} c^{in} < 0 \,,
\qquad
  x_- = x - \tau_-(x, v) v \,.
\end{align}
Application of Lemma~\ref{lem:tau-continuity} gives that
\begin{align*}
   \tau_-(x, v) \in C^1(\bar D_{\phi_0}) \,,
\end{align*}
which implies that $\tau_-(x, v)$ is uniformly continuous on $D_{\phi_0}$. 
Together with the continuity of $\sigma_a$, and $\phi_0$, we deduce that for each $\Eps$, the function $H_\Eps(\cdot, \cdot): \bar D_{\phi_0} \to \R$ is continuous. Hence $H_\Eps$ is uniformly continuous on $\bar D_{\phi_0}$ and thus
\begin{align*}
   \frac{1}{\delta} \int_{\Ss^1_{x, +}} 
   \abs{H_\Eps(x, v) - H_\Eps(x, \Vin)}
   \psi_0 \vpran{\frac{\abs{v - \Vin}}{\delta}}
   \phi_0 \vpran{\frac{\abs{v - \Vin}}{\delta}} \dv
\to 0 \qquad\text{as $\delta \to 0$ uniformly in $x$.}
\end{align*}
Therefore, for each $x \in \bar\Omega$, 
\begin{align*}
   \lim_{\delta \to 0} G_{\Eps, \delta}(x)
&= H_\Eps(x, \Vin) 
    \vpran{\frac{1}{\delta} \lim_{\delta \to 0} \int_{\Ss^1_{x,+}}\psi_0 \vpran{\frac{\abs{v - \Vin}}{\delta}}
   \phi_0 \vpran{\frac{\abs{v - \Vin}}{\delta}} \dv}
\to C_{\psi_0, \phi_0} H_\Eps(x, \Vin) \,,
\end{align*}
where the constant $C_{\psi_0, \phi_0}$ is given by
\begin{align*}
   C_{\psi_0, \phi_0}
= \frac{1}{\delta} \lim_{\delta \to 0} \int_{\Ss^1_{x,+}}\psi_0 \vpran{\frac{\abs{v - \Vin}}{\delta}}
\phi_0 \vpran{\frac{\abs{v - \Vin}}{\delta}} \dv
= \int_\R \phi_0(r) \psi_0(r) \dr \,.
\end{align*}
Applying the Lebesgue Dominated Convergence Theorem we obtain
\begin{align*}
& \quad \,
     \lim_{\delta \to 0} M_\psi(f_1)\\
&= \frac{C_{\psi_0, \phi_0}}{\Eps}
      \int_{\del\Omega}   \psi_0 (x - \Xout)  \,
     H_\Eps(x, \Vin) \dS_x
\\
& = \frac{C_{\psi_0, \phi_0}}{\Eps}
      \int_{\del\Omega}   \psi_0 (x - \Xout)  \,
      e^{-\int_0^{\tau_-(x, \Vin)} \sigma_a(x - s \Vin) \ds}
    \phi_0 \vpran{\frac{\abs{x - \tau_-(x, \Vin) \Vin - \Xin}}{\Eps}}
    n(x) \cdot \Vin \dS_x \,.
\end{align*}
Furthermore if we make the change of variables using
\begin{align*}
    y = x_-(x, \Vin) = x - \tau_-(x, \Vin) \Vin \,,
\end{align*}
then by the non-degeneracy in~\eqref{cond:non-degen-1}, the mapping is invertible and we claim that
\begin{align} \label{geometry-1}
   \abs{n(y) \cdot \Vin} \dS_y 
= - n(y) \cdot \Vin \dS_y
= \abs{n(x) \cdot \Vin} \dS_x
= n(x) \cdot \Vin \dS_x\,.
\end{align}
This relation can be justified through the physical meanings of $\abs{n(x) \cdot \Vin} \dS_x$ and $\abs{n(y) \cdot \Vin} \dS_y$  as the effective fluxes into and out of the boundary. The mathematical proof for~\eqref{geometry-1} is given in Appendix~\ref{sec:geometry}. 
Making such change of variables, we obtain that
\begin{align}
& \quad \,
   \lim_{\Eps \to 0} \lim_{\delta \to 0} M_\psi(f_1)
\\
& = \lim_{\Eps \to 0} \frac{C_{\psi_0, \phi_0}}{\Eps}
      \int_{\del\Omega}   \psi_0 (x(y) - \Xout)  \,
      e^{-\int_0^{\tau_-(x(y), \Vin)} \sigma_a(x(y) - s \Vin) \ds}
    \phi_0 \vpran{\frac{\abs{y - \Xin}}{\Eps}}
    \abs{n(y) \cdot \Vin} \dS_y \nn
\\
& = C_{\psi_0, \phi_0} \abs{n(\Xin) \cdot \Vin}
   e^{-\int_0^{\tau_-(\Xout, \Vin)} \sigma_a(\Xout - s \Vin) \ds} \nn
\\
&   = C_{\psi_0, \phi_0} \abs{n(\Xin) \cdot \Vout}
       e^{-\int_0^{\tau_-(\Xout, \Vout)} \sigma_a(\Xout - s \Vout) \ds} \,,
       \label{eq:X-ray-f-1-RTE-general}
\end{align}
where we have applied the differential relation $\dS_y = {\rm d}|y - x^{in}|$ and $\psi_0(0) = 1$. Note that the last term involves the X-ray transformation of $\sigma_a$.

\smallskip
\noindent
\underline{\textbf{Step 2: limit of $M_\psi(f_2)$}} \,
By~\eqref{eqn:f2_soln}, the contribution of $f_2$ toward the measurement is
\begin{align} \label{formula:f-2-contri-RTE-general}
M_\psi(f_2)= \int_{\Ss^1} \!\! \int_{\del\Omega_{v}^+} \!\!
    \int_0^{\tau_-(x, v)} \psi(x, v) e^{-\int_0^s \sigma_a(x - \tau v) \dtau}
     F_f(x - sv) n(x) \cdot v\ds \dS_x \dv  \,.
\end{align}
Make a change of variables in the above integral with
\begin{align} \label{def:change-variable-1}
    y = x - s v \,,\qquad \text{for}\quad x \in \del\Omega_v^+\quad\text{and}\quad s \in (0, \tau_-(x, v))\,.
\end{align}
Note that $y \to (x, s)$ is a one-to-one mapping with the relation (verified in Appendix~\ref{sec:geometry})
\begin{align} \label{geometry-2}
    \dy = n(x) \cdot v \ds\dS_x
\end{align}
and the inverse map is
\begin{align*}
   s = \tau_+(y, v) \,,
\qquad
   x = y + \tau_+(y, v) v \,.
\end{align*}
Hence one can rewrite the integral in~\eqref{formula:f-2-contri-RTE-general} as
\begin{align*}
M_\psi(f_2) = \int_{\Ss^1} \int_{\Omega}
   \psi(y + \tau_+(y, v) v, v)
   e^{-\int_0^{\tau_+(y, v)} \sigma_a(y + \tau_+(y, v) v - \tau v) \dtau} 
   F_f(y) \dy \dv \,.
\end{align*}
With the definition of $\psi$ and the bound for $F_f$ in~\eqref{bound:RTE-general}, we obtain that
\begin{align} 
   \abs{M_\psi(f_2)} \nn &\leq
  \int_{\Ss^1} \int_{\Omega} \psi_0 \vpran{\frac{v - \Vin}{\delta}}\abs{F_f(y)} \dy \dv \nn
\\
&= \vpran{\int_{\Ss^1} \psi_0 \vpran{\frac{v - \Vin}{\delta}} \dv}
     \vpran{\int_{\Omega} \abs{F_f(y)} \dy} \nn
\\
&\leq
  C_\Omega \delta \norm{F_f}_{L^p(\Omega)}
\leq
  C_\Omega \delta \norm{\phi}_{L^p(\Gamma_-)} \,. \label{bound:f-2-intermediate}
\end{align}
The $L^p$-norm of $\phi$ can be estimated using its definition:
\begin{align*}
  \norm{\phi}_{L^p(\Gamma_-)}^p
&= \iint_{\Gamma_-} \phi^p (x, v) \abs{n(x) \cdot v} \dS_x \dv
\\
&= \frac{1}{\Eps^p \delta^p} \iint_{\Gamma_-} 
   \phi_0^p\vpran{\frac{|x - \Xin|}{\Eps}}
   \phi_0^p\vpran{\frac{|v - \Vin|}{\delta}} \abs{n(x) \cdot v} \dS_x \dv
\\
& \leq
   \frac{1}{\Eps^p \delta^p} 
   \vpran{\int_{\del\Omega}\phi_0\vpran{\frac{|x - \Xin|}{\Eps}} \dS_x}
   \vpran{\int_{\Ss^1} \phi_0\vpran{\frac{|v - \Vin|}{\delta}}\dv}
\\
& \leq 
   C_{\phi_0} \Eps^{-(p-1)} \delta^{-(p-1)} \,.
\end{align*}
Plugging such bound back in~\eqref{bound:f-2-intermediate} we obtain
\begin{align}\label{eqn:f2_estimate}
     \lim_{\Eps \to 0} \lim_{\delta \to 0}\abs{ M_\psi(f_2)} \leq
  C_{\Omega, \phi_0}   \lim_{\Eps \to 0} \lim_{\delta \to 0}\Eps^{-\frac{p-1}{p}} \delta^{\frac{1}{p}} =0\,.
\end{align}
Finally, by combining~\eqref{eq:X-ray-f-1-RTE-general} and~\eqref{eqn:f2_estimate} we have
\begin{align*}
    \lim_{\Eps \to 0} \lim_{\delta \to 0} M_\psi(f) = \iint_{\Gamma_+} \psi(x, v) f(x, v) \dGamma
= C_{\phi_0, \psi_0} \abs{n(\Xin) \cdot \Vin} e^{-\int_0^{\tau_-(\Xout, \Vout)} \sigma_a(\Xout - s \Vout) \ds} \,.
\end{align*}
Therefore, the X-ray transformation of $\sigma_a$ is uniquely determined by the measurement, which in turn implies that $\sigma_a$ is uniquely recoverable by the measurement.  
\end{proof}

\subsection{Examples}
Theorem~\ref{thm:RTE-general} is rather general and one only needs to verify two conditions in order to apply it: the well-posedness of the forward problem and the bound~\eqref{bound:RTE-general}. For many kinetic equations these conditions follow from energy methods. Below we give two examples.

The first example is the classical linear RTE with $F_f = \sigma_s\vint{f}$ and the equation reads
\begin{align} \label{eq:RTE-linear}
   v \cdot \nabla_x f = - \sigma_a f + \sigma_s \vint{f} \,.
\end{align}
The statement of the unique solvability of $\sigma_a$ is
\begin{thm} \label{thm:RTE-linear}
Suppose $\Omega$ is a strictly convex and bounded domain with a $C^1$ boundary. Suppose there exists a constant $\sigma_0 > 0$ such that
\begin{align*}
    \sigma_a \in C(\bar \Omega) \,,
\qquad 
   \sigma_a \geq \sigma_s \geq \sigma_0 > 0 \,.
\end{align*}
Then with proper choices of the incoming data, the absorption coefficient $\sigma_a$ can be uniquely recovered from the measurement of the outgoing data. 
\end{thm}
This is the example studied in the original singular decomposition work~\cite{CS98} where the subcritical case with $\sigma_a - \sigma_s > 0$ is considered. We are now able to treat the critical and subcritical cases with $\sigma_a \geq \sigma_s$ in a unified way.

\begin{proof}
Let $\phi$ be a nonnegative incoming data such that $\phi \in L^2(\Gamma_-)$. Then the positivity of $f$ follows from the maximum principle of the linear RTE and the unique solvability is classical \cite{DL}.
In equation~\eqref{eq:RTE-linear} we have $F_f(x) = \sigma_s \vint{f}$.
To obtain an $L^2$-bound of $F_f$, multiply~\eqref{eq:RTE-linear} by $2 f$ and integrate in $(x, v)$. This gives
\begin{align*}
   \int_\Omega \int_{\Ss^1} v \cdot \nabla_x f^2
&= -2 \int_\Omega \int_{\Ss^1} \sigma_s \vpran{f - \vint{f}}^2 
   - 2 \int_\Omega \int_{\Ss^1} \vpran{\sigma_a - \sigma_s} f^2 
\leq
   -2 \sigma_0 \norm{f - \vint{f}}_{L^2(\Omega \times \Ss^1)}^2 \,.
\end{align*}
By integration by parts, the left-hand satisfies
\begin{align*}
   \int_\Omega \int_{\Ss^1} v \cdot \nabla_x f^2
= \int_{\del\Omega}\int_{\Ss^1} (n(x) \cdot v) f^2 
\geq
   \iint_{\Gamma_-} (n(x) \cdot v) f^2 
= - \norm{\phi}_{L^2(\Gamma_-)}^2 \,.
\end{align*}
Combining the above two inequalities we have
\begin{align} \label{bound:ortho-RTE-linear}
   \norm{f - \vint{f}}_{L^2(\Omega \times \Ss^1)}^2
\leq
  \frac{1}{2 \sigma_0} \norm{\phi}_{L^2(\Gamma_-)}^2 \,.
\end{align}
Denote $g = f - \vint{f}$. Since $f \geq 0$, we have
\begin{align} \label{eq:RTE-g}
   v \cdot \nabla f \leq -\sigma_s g \,.
\end{align}
Solving along charateristics, we have
\begin{align*}
   f(x + t v, v)
\leq
   \phi(x, v) - \int_0^t \sigma_s(x + \tau v) g(x + \tau v, v) \dtau \,,
\qquad
   (x, v) \in \Gamma_- \,,
\,\,
   t \in [0, \tau_+(x, v)] \,.
\end{align*}
Hence, for any $   (x, v) \in \Gamma_-$ and $t \in [0, \tau_+(x, v)]$, it holds that
\begin{align*}
   f^2(x + t v, v)
&\leq
   2 \phi^2(x, v) + 2 \vpran{\int_0^{\tau_+(x, v)} \sigma_s(x + \tau v) \abs{g(x + \tau v, v)} \dtau}^2
\\
&\leq
   2 \phi^2(x, v) 
+ 2 \vpran{\text{diam}(\Omega)} \norm{\sigma_s}_{L^\infty(\Omega)}^2
   \int_0^{\tau_+(x, v)} g^2(x + \tau v, v) \dtau \,.
\end{align*}
Integrating in $(x, v) \in \Gamma_-$ and $t \in [0, \tau_+(x, v)]$, we obtain that
\begin{align*}
  \iint_{\Gamma_-} \int_0^{\tau_+(x, v)} 
  f^2(x+tv, v) \dt {\rm d}\Gamma_-
&\leq
  2 \vpran{\text{diam}(\Omega)} \norm{\phi}_{L^2(\Gamma_-)}^2
\\
& \quad \,  
+ 2 \vpran{\text{diam}(\Omega)}^2 \norm{\sigma_s}_{L^\infty(\Omega)}^2
   \iint_{\Gamma_-}\int_0^{\tau_+(x, v)} g^2(x + \tau v, v) \dtau {\rm d}\Gamma_- \,.
\end{align*}
Using a similar changing of variables as in~\eqref{def:change-variable-1} by letting $z = x + \tau v$, we then derive that
\begin{align*}
  \norm{f}_{L^2(\Omega \times \Ss^1)}^2
\leq
  C_\Omega \norm{\phi}_{L^2(\Gamma_-)}^2
  + C_\Omega \norm{g}_{L^2(\Omega \times \Ss^1)}^2
\leq
  C_\Omega \norm{\phi}_{L^2(\Gamma_-)}^2 \,,
\end{align*}
where the last inequality follows from~\eqref{bound:ortho-RTE-linear}. 
Hence, we derive that
\begin{gather*}
   \norm{F_f}_{L^2(\Omega \times \Ss^1)}
= \norm{\sigma_s \vint{f}}_{L^2(\Omega \times \Ss^1)}
\leq
   \norm{\sigma_s}_{L^\infty(\Omega)} \norm{f}_{L^2(\Omega \times \Ss^1)}
\leq
  C_\Omega \norm{\phi}_{L^2(\Gamma_-)}^2 \,, 
\end{gather*}
which combined with Theorem~\ref{thm:RTE-general} gives the desired unique solvability of $\sigma_a$.
\end{proof}

In the second example we consider a nonlinear RTE, which couples the temperature and the intensity of the rays. The equation has the form \cite{Klar_nonlinear_RTE}:
\begin{align}
   v \cdot \nabla_x I &= -\sigma_a I + \sigma_a T^4 \,, \label{eq:RTE-nonlinear-1}
\\
   \Delta_x T &= \sigma_a T^4 - \sigma_a \vint{I} \,. \label{eq:RTE-nonlinear-2}
\end{align}
The statement of the unique solvability of $\sigma_a$ in~\eqref{eq:RTE-nonlinear-1}-\eqref{eq:RTE-nonlinear-2} is
\begin{thm} \label{thm:nonlinear-RTE}
Suppose $\Omega$ is a convex and bounded domain with a $C^1$ boundary. Suppose there exists a constant $\sigma_0 > 0$ such that
\begin{align*}
   \sigma_a \geq \sigma_0 > 0 \,,
\qquad
    \sigma_a \in C(\bar \Omega) \,,
\end{align*}
Then with proper choices of the incoming data, the absorption coefficient $\sigma_a$ can be uniquely recovered from the measurement of the outgoing data. 
\end{thm}
\begin{proof}
Given an incoming data $\phi$ for $I$ and a zero boundary condition $T_B$ for $T$, we show the well-posedness of~\eqref{eq:RTE-nonlinear-1}-\eqref{eq:RTE-nonlinear-2} in Appendix~\ref{sec:appendix-1}.  The non-negativity of $I$ follows directly from the observation that $\sigma_a T^4 \geq 0$. Now we have $F_f = \sigma_a T^4$ and we want to show that there exists a constant $C_0$ such that
\begin{align} \label{bound:F-f-RTE-nonlinear}
    \norm{\sigma_a T^4}_{L^2(\Omega)} 
 \leq
    C_0 \norm{\phi}_{L^2(\Gamma_-)} \,.
\end{align}
Such $L^2$-bound can be obtained by the energy method along a similar line as in \cite{Klar_nonlinear_RTE}. For the convenience of the reader we include the details here. The full equation with the boundary conditions reads
\begin{align}
   v \cdot \nabla_x I &= -\sigma_a I + \sigma_a T^4 \,, 
\qquad
   I \big|_{\Gamma_-} = \phi(x, v) \,,
   \label{eq:RTE-nonlinear-3}
\\
   \Delta_x T &= \sigma_a T^4 - \sigma_a \vint{I},
\qquad
   T \big|_{\del\Omega} = 0 \,.
   \label{eq:RTE-nonlinear-4}
\end{align} 
Multiply~\eqref{eq:RTE-nonlinear-3} by $I$ and~\eqref{eq:RTE-nonlinear-4} by $T^4$. Then integrate both equations in $(x, v)$ and take their difference. By rearranging terms we get
\begin{align*}
   \frac{1}{2}\int_{\del \Omega} \int_{\Ss^1} (n(x) \cdot v) I^2 
   + 4 \int_\Omega T^3 \abs{\nabla_x T}^2 \nn
&= - \int_\Omega \int_{\Ss^1} \sigma_a I^2 
   + 2\int_\Omega \int_{\Ss^1} \sigma_a \vint{I} T^4 
   -  \int_\Omega \int_{\Ss^1} \sigma_a T^8\,.
\\
& = - \int_\Omega \int_{\Ss^1} \sigma_a 
      \vpran{\vint{I} - T^4}^2
      - \int_\Omega \int_{\Ss^1} \sigma_a
         \vpran{I^2 - \vint{I}^2} \,.
\end{align*}
where it has been shown in Theorem~\ref{thm:well-posed-nonlinear-RTE} that $T \geq 0$ given $\phi$ non-negative. Dropping the term involving $T^3$, we have
\begin{align} \label{est:RTE-nonlinear-1}
   \sigma_0 \norm{\vint{I} - T^4}_{L^2(\Omega)}^2
\leq
   \int_\Omega \int_{\Ss^1} \sigma_a 
      \vpran{\vint{I} - T^4}^2
\leq
   \frac{1}{2} \norm{\phi}_{\Gamma_-}^2
\end{align}
and
\begin{align} \label{est:RTE-nonlinear-2}
      \sigma_0 \norm{I - \vint{I}}_{L^2(\Omega)}^2
\leq \int_\Omega \int_{\Ss^1} \sigma_a 
      \vpran{I - \vint{I}}^2
= \int_\Omega \int_{\Ss^1} \sigma_a 
      \vpran{I^2 - \vint{I}^2}
\leq
   \frac{1}{2} \norm{\phi}_{\Gamma_-}^2 \,.
\end{align}
Combining~\eqref{est:RTE-nonlinear-1} with~\eqref{est:RTE-nonlinear-2}, we obtain that
\begin{align*}
   \norm{\sigma_a (I - T^4)}_{L^2(\Omega \times \Ss^1)}
\leq
  \norm{\sigma_a}_{L^\infty}
  \vpran{\norm{I - \vint{I}}_{L^2(\Omega \times \Ss^1)}
  + \norm{\vint{I} - T^4}_{L^2(\Omega \times \Ss^1)}}
\leq
  \frac{\norm{\sigma_a}_{L^\infty}}{\sigma_0}
  \norm{\phi}_{\Gamma_-}^2 \,.
\end{align*}
Since $\sigma_a (T^4 - I)$ is simply the forcing term in~\eqref{eq:RTE-nonlinear-3}, we can apply the bound for~\eqref{eq:RTE-g} to derive that 
\begin{align*}
   \norm{I}_{L^2(\Omega \times \Ss^1)}
\leq
  C_\Omega \norm{\phi}^2_{\Gamma_-} \,.
\end{align*}
This implies that
\begin{align*}
   \norm{\sigma_a T^4}_{L^2(\Omega)}
\leq
  \norm{\sigma_a}_{L^\infty}
  \vpran{\norm{I}_{L^2(\Omega \times \Ss^1)}
  + \norm{I - T^4}_{L^2(\Omega \times \Ss^1)}}
\leq
  C_\Omega(\norm{\sigma_a}_{L^\infty}, \sigma_0) \norm{\phi}_{\Gamma_-}^2 \,,
\end{align*}
which is the desired bound in~\eqref{bound:F-f-RTE-nonlinear}. The unique solvability of $\sigma_a$ then follows from Theorem~\ref{thm:RTE-general}.
\end{proof}

\section{Recovery of the Scattering Coefficient: Averaging Lemma} \label{sec:scattering}
In this section, we show how to use the celebrated averaging lemma for kinetic equations to recover the scattering coefficient. We will work out a specific example as an illustration. The equation under consideration is
~\eqref{eq:RTE-linear}, which we recall as
\begin{align} \label{eq:RTE-linear-recall}
   v \cdot \nabla_x f = - \sigma_a f + \sigma_s \vint{f} \,.
\end{align}
Since $\sigma_a$ has been found by Theorem~\ref{thm:RTE-linear}, in what follows we assume that $\sigma_a$ is given and focus on finding~$\sigma_s$.

First we recall the statement of the averaging lemma. For the purpose of the current work, we only need the most basic version which is stated as
\begin{thm}[\cite{averaging,averaging_Lions,Bezard_averaging}]\label{thm:averaging}
Suppose $0 < \sigma_0 \leq \sigma_s \leq \sigma_a$ with $\sigma_a \in C(\bar\Omega)$ where $\Omega$ is open and bounded. Suppose $\phi \in L^p(\Gamma_-)$ and $g \in L^p(\Omega \times \Ss^1)$ for some $p > 1$ and $f$ satisfies the equation
\begin{align} \label{eq:RTE-linear-full-recall}
   v \cdot \nabla_x f = - \sigma_a f + \sigma_s \vint{f} + g \,,
\qquad
  f \big|_{\Gamma_-} = \phi(x, v) \,.
\end{align}
Then for any $\gamma \leq \inf \{\frac{1}{p}, 1 - \frac{1}{p}\}$, the velocity average of $f$ satisfies $\vint{f} \in W^{\gamma, p}(\Omega)$ with the bound
\begin{align*}
   \norm{\vint{f}}_{W^{\gamma, p}(\Omega)}
\leq
   C_0 \vpran{\norm{\phi}_{L^p(\Gamma_-)} 
                      + \norm{g}_{L^p(\Omega \times \Ss^1)}} \,.
\end{align*}
\end{thm}

We also recall the basic $L^p$ energy estimate \cite{Egger_Lp} for equation~\eqref{eq:RTE-linear-full-recall}:
\begin{thm}[\cite{Egger_Lp}]\label{thm:L-p}
Suppose $\phi \in L^p(\Gamma_-)$ and $g \in L^p(\Omega \times \Ss^1)$ for some $p \in [1, \infty]$. Then $f \in L^p(\Omega \times \Ss^1)$ with the bound
\begin{align*}
    \norm{f}_{L^p(\Omega \times \Ss^1)}
\leq
   C_0 \vpran{\norm{\phi}_{L^p(\Gamma_-)} 
                      + \norm{g}_{L^p(\Omega \times \Ss^1)}} \,.
\end{align*}

\end{thm}

Our main result in this part is
\begin{thm}\label{thm:recover-scattering}
Let $\Omega \subseteq \R^2$ be a strictly convex and bounded domain with a $C^1$ boundary. Suppose $0 < \sigma_0 \leq \sigma_s \leq \sigma_a$ with $\sigma_a \in C(\bar\Omega)$ given.
Then with proper choices of the incoming data, the scattering coefficient $\sigma_s$ in~\eqref{eq:RTE-linear-recall} can be uniquely recovered from the measurement of the outgoing data.  
\end{thm}
\begin{proof}
For any given $\phi$, let $f$ be the solution to~\eqref{eq:RTE-linear-full-recall}. Decompose it into three parts: $f = f_1 + f_2 + f_3$, where 
\begin{align} 
   v \cdot \nabla_x f_1 = - \sigma_a f_1 \,,
\qquad
  f_1 \big|_{\Gamma_-} = \phi(x, v) \,, \label{eq:scattering-1}
\end{align}
\begin{align}  \label{eq:scattering-2}
   v \cdot \nabla_x f_2 = - \sigma_a f_2 + \sigma_s \vint{f_1} \,,
\qquad
  f_2 \big|_{\Gamma_-} = 0 \,,
\end{align}
\begin{align} \label{eq:scattering-3}
   v \cdot \nabla_x f_3 
= - \sigma_a f_3 + \sigma_s \vint{f_2} + \sigma_s \vint{f_3}\,,
\qquad
  f_3 \big|_{\Gamma_-} = 0 \,.
\end{align}
Note that given $\sigma_a, \sigma_s$, the first two functions $f_1, f_2$ are explicitly solvable. The idea of the proof is to show $f_3$ is more regular than $f_2$, which in turn more regular than $f_1$, using the averaging lemma. By posing the correct geometry for the incoming and measuring functions, one can show $f_2$ dominates the data, and is used to reconstruct $\sigma_s$.

\medskip

\Ni \underline{\bf Incoming and Measurement} \,
First we need to specify the incoming data $\phi$ and the measurement function $\psi$. Fix $(\Xin, \Vin) \in \Gamma_-$ and $(\Xout, \Vout) \in \Gamma_+$ such that 
\begin{align} \label{cond:V-in-out}
   \Vin \nparallel \Vout \,,
\qquad
  \Vin \cdot \Vout > 0 \,.
\end{align} 
Let $\ell_1$ be the ray initiated at $\Xin$ along the direction $\Vin$ and $\ell_2$ the ray initiated at $\Xout$ along the direction $-\Vout$. Since $\Vin \nparallel \Vout$, the two rays $\ell_1$ and $\ell_2$ have a unique intersection inside $\Omega$, which we denote as $x_0$. For later use, let $s_0 > 0$ be the exit time associated with $x_0$ in the direction of $\Vout$, or more explicitly,
\begin{align} \label{def:x-0}
   x_0 = \Xout - s_0 \Vout = \Xin + s'_0 \Vin \,.
\end{align}
The main goal is to find $\sigma_s(x_0)$. Define $\Vin_\perp$ as the unit vector such that 
\begin{align} \label{def:eta}
   \Vin_\perp \cdot \Vin = 0 \,,
\quad \text{and} \quad 
  \eta  := \Vin_\perp \cdot \Vout > 0 \,.
\end{align}
For the illustration of the geometry, see Figure~\ref{fig:figure1}.
\begin{figure}[h]
\includegraphics[scale = 0.36]{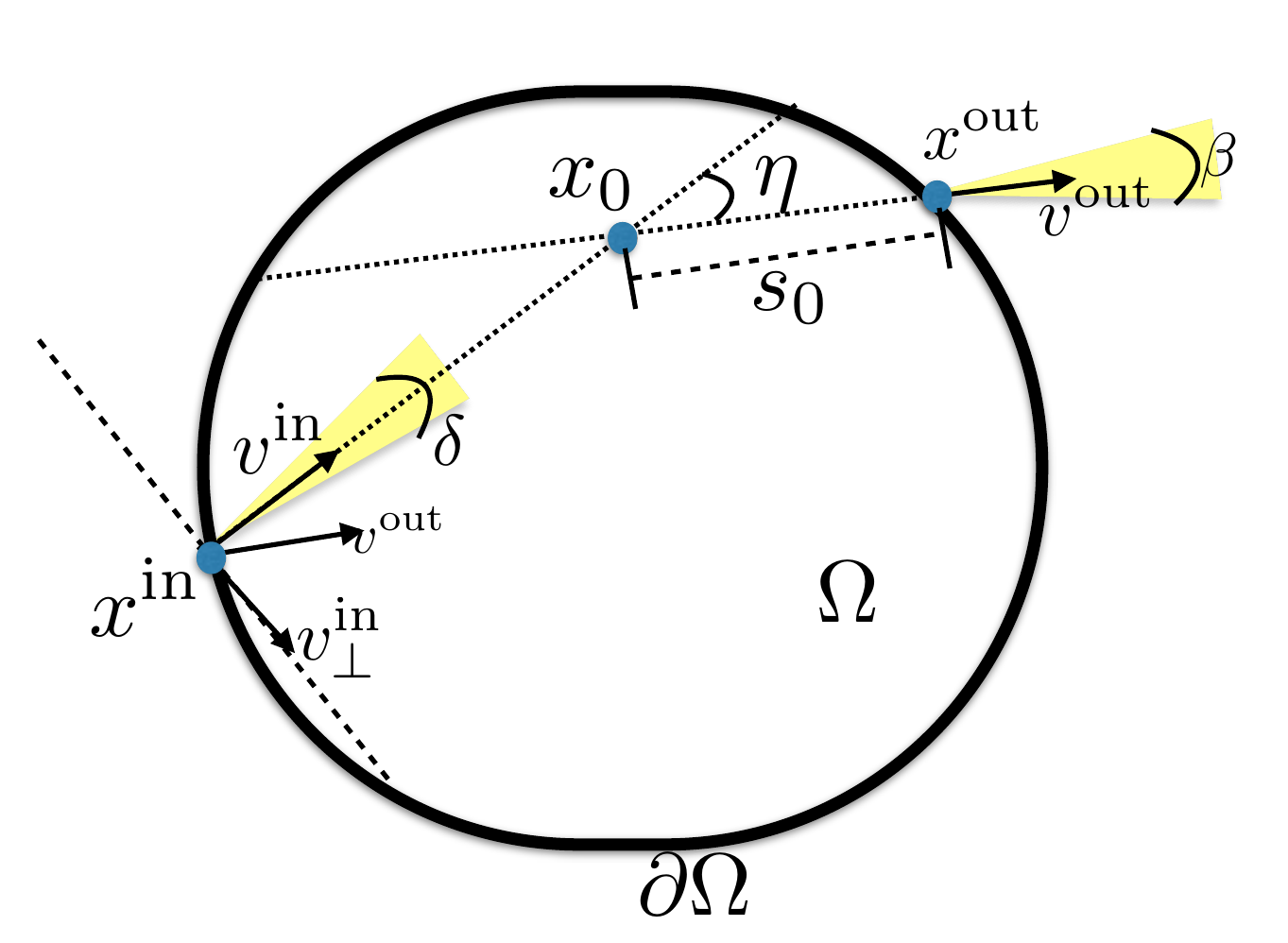}
\caption{Geometry and some physical quantities.}\label{fig:figure1}
\end{figure}

Let $\phi_0$ be a smooth even function on $\R$ such that 
\begin{align*}
  0 \leq \phi_0(r) \leq 1 \,,
\qquad
  \overline{\Supp\phi_0} = [-1, 1] \,,
\qquad 
  \phi_0(0) = 1\,,
\qquad
  \int_\R \phi_0(r) \dr = 1 \,.
\end{align*}
Let $\psi_0$ be the same smooth function defined in~\eqref{def:psi-0} with $\overline{\Supp\psi_0} = [-1, 1]$.
We choose the incoming data $\phi$ and the measurement function $\psi$ as
\begin{align*}
   \phi(x, v) 
&= \frac{1}{\Eps \delta}  
    \phi_0 \vpran{\frac{(x - \Xin) \cdot \Vin_\perp}{\Eps \eta}}
    \phi_0 \vpran{\frac{|v - \Vin|}{\delta}} \,, 
\qquad
   (x, v) \in \Gamma_- \,,
\\
   \psi(x, v)
& = \frac{1}{\theta \beta}
   \psi_0 \vpran{\frac{|x - \Xout|}{\theta}}
   \psi_0 \vpran{\frac{|v - \Vout|}{\beta}} \,,
\hspace{1.2cm}
   (x, v) \in \Gamma_+ \,.
\end{align*}
Quickly, we have
\[
\CalM_\psi(f) = \iint_{\Gamma_-}\psi f\rd\Gamma_-=\CalM_\psi(f_1)+\CalM_\psi(f_2)+\CalM_\psi(f_3)\,.
\]
The essence of the proof is to show that $\CalM_\psi(f_1)$ and $\CalM_\psi(f_3)$ are negligible while $\CalM_\psi(f_2)$ is used to reconstruct $\sigma_s(x_0)$. The estimate for $\CalM_\psi(f_3)$ relies on the averaging lemma, and the estimate for $\CalM_\psi(f_1)$ follows from a basic geometric argument.

As a preparation, we first give an estimate of $L^r$-bound of $\phi$ (with $r$ to be determined later):
\begin{align*}
  \iint_{\Gamma_-} \phi^r (x, v) |v \cdot n| \dS_x \dv
&= \frac{1}{\Eps^r \delta^r}
   \int_{\del\Omega} \int_{\Ss^1_{x, +}}
   \phi_0^r \vpran{\frac{(x - \Xin) \cdot \Vin_\perp}{\Eps \eta}}
   \phi_0^r \vpran{\frac{|v - \Vin|}{\delta}}  |v \cdot n| \dS_x \dv
\\
&\leq
  \vpran{\frac{1}{\delta^r} \int_{\Ss^1} \phi_0^r \vpran{\frac{|v - \Vin|}{\delta}} \dv}
  \vpran{\frac{1}{\Eps^r} \int_{\del\Omega} \phi_0^r \vpran{\frac{(x - \Xin) \cdot \Vin_\perp}{\Eps \eta}} \dS_x} \,,
\end{align*}
where the $v$-integral is bounded as
\begin{align*}
   \frac{1}{\delta^r} \int_{\Ss^1} \phi_0^r \vpran{\frac{|v - \Vin|}{\delta}} \dv
= \frac{1}{\delta^r} \int_0^{2\pi} \phi_0^r \vpran{\frac{|\sin\omega/2|}{\delta/2}} \domega
\leq 
  \frac{1}{\delta^r} \int_0^{2\pi} \phi_0 \vpran{\frac{|\sin\omega/2|}{\delta/2}} \domega
\leq
  c_0 \delta^{-(r-1)} \,.
\end{align*}
In order to estimate the boundary integral, we take $\Vin_\perp$ as the horizontal axis and take $\Eps \eta$ small enough such that $\del \Omega$ is a graph parametrized by
\begin{align*}
   x_2 = f(x_1) \,,
\qquad
   x_1 \in (\Xin_1 - \Eps \eta, \,\, \Xin_1 + \Eps \eta) \,,
\qquad
   x = (x_1, x_2) \,,
\end{align*}
where $f  \in C^1 [\Xin_1 - h_0, \,\, \Xin_1 + h_0]$ for some fixed $h_0$. Then the boundary integral satisfies
\begin{align*}
   \frac{1}{\Eps^r} \int_{\del\Omega} \phi_0^r \vpran{\frac{(x - \Xin) \cdot \Vin_\perp}{\Eps \eta}} \dS_x
&= \frac{1}{\Eps^r} \int_{\Xin_1 - \Eps \eta}^{\Xin_1 + \Eps \eta} \phi_0^r \vpran{\frac{x_1 - \Xin_1}{\Eps \eta}} \sqrt{1 + |f'(x_1)|^2} \dx_1 
\\
& \leq
   \frac{c_0}{\Eps^r} \int_{\Xin_1 - \Eps \eta}^{\Xin_1 + \Eps \eta} \phi_0 \vpran{\frac{x_1 - \Xin_1}{\Eps \eta}} \dx_1
\leq 
   c_0 \Eps^{-(r-1)} \eta \,,
\end{align*}
where $c_0$ depends on the $C^1$-norm of $f$, which is assumed to be bounded since $\del\Omega$ is $C^1$. Note that such bound is independent of $\Xin$ since $\del\Omega$ is compact. Combining the two integrals, we have
\begin{align*} 
  \norm{\phi}_{L^r(\Gamma_-)}
\leq
  \vpran{\iint_{\Gamma_-} \phi^r (x, v) |v \cdot n| \dS_x \dv}^{1/r}
\leq
  c_0 \Eps^{-\frac{r-1}{r}} \delta^{-\frac{r-1}{r}} \eta^{\frac{1}{r}} \,,
\qquad
r > 1 \,. 
\end{align*}

\Ni \underline{\bf Averaging Lemma} \, Now we apply the $L^r$-energy bound and the averaging lemma to obtain a bound for $\vint{f_1}$, $\vint{f_2}$, and $f_3$. First, a direct application of Theorem~\ref{thm:averaging} gives
\begin{align*}
   \norm{\vint{f_1}}_{W^{s_0, r}}
\leq
   c_0 \norm{\phi}_{L^r(\Gamma_-)}
\leq
   c_0 \Eps^{-\frac{r-1}{r}} \delta^{-\frac{r-1}{r}} \eta^{\frac{1}{r}} \,,
\end{align*}
where $s_0 = \inf \{\frac{1}{r}, 1 - \frac{1}{r}\}$. By the Sobolev embedding, we have
\begin{align*}
   \norm{\vint{f_1}}_{L^{p_1}(\Omega)}
\leq
  c_0 \norm{\vint{f_1}}_{W^{s_0, r}}
\leq
  c_0 \, \Eps^{-\frac{r-1}{r}} \delta^{-\frac{r-1}{r}} \eta^{\frac{1}{r}} \,,
\qquad
  \forall \, p_1 \leq \frac{1}{\frac{1}{r} - \frac{s_0}{2}} \,.
\end{align*}
Since $\vint{f_1}$ is the source term in the equation for $f_2$, we apply the averaging lemma again and get
\begin{align} \label{bound:vint-f-2}
   \norm{\vint{f_2}}_{L^{p_2}(\Omega)}
\leq
   c_0 \norm{\vint{f_2}}_{W^{s_1, p_1}(\Omega)}
\leq
  c_0 \norm{\vint{f_1}}_{L^{p_1}(\Omega \times \Ss^1)}
\leq
  c_0 \Eps^{-\frac{r-1}{r}} \delta^{-\frac{r-1}{r}} \eta^{\frac{1}{r}} \,,
\end{align}
where the exponents satisfy that
\begin{align*}
   s_1 = \inf \left\{\frac{1}{p_1}, \,\, 1 - \frac{1}{p_1} \right\} \,,
\qquad
   p_2 \leq \frac{1}{\frac{1}{p_1} - \frac{s_1}{2}} \,.
\end{align*}
By Theorem~\ref{thm:L-p}, we also have
\begin{align} \label{bound:vint-f-3}
   \norm{\vint{f_3}}_{L^{p_2}(\Omega)}
\leq
   \norm{f_3}_{L^{p_2}(\Omega \times \Ss^1)}
\leq
   c_0 \norm{\vint{f_2}}_{L^{p_2}(\Omega)} \,.
\end{align}

\Ni \underline{\bf Contribution from $f_3$} \, Using the change of variables in~\eqref{geometry-2}, we obtain the contribution of $f_3$ to the measurement of the outgoing data as
\begin{align*}
   \abs{ \iint_{\Gamma_+} \psi(x, v) f_3(x, v) \dGamma} \nn
&\leq
  \norm{\sigma_s}_{L^\infty} \int_{\Ss^1} \int_{\Omega} \psi(y + \tau_+(y, v) v, v) \abs{\vint{f_2}(y) + \vint{f_3}(y)} \dy \dv
\\
& \hspace{-2.5cm}
\leq
  c_0 \int_{\Ss^1}\int_{\Omega} \frac{1}{\theta \beta}
   \psi_0 \vpran{\frac{|y + \tau_+(y, v) v - \Xout|}{\theta}}
   \psi_0 \vpran{\frac{|v - \Vout|}{\beta}}
   \abs{\vint{f_2}(y) + \vint{f_3}(y)} \dy \dv 
\\
& \hspace{-2.5cm}
= c_0 \int_\Omega 
   \vpran{\int_{\Ss^1} \frac{1}{\theta \beta}
   \psi_0 \vpran{\frac{|y + \tau_+(y, v) v - \Xout|}{\theta}}
   \psi_0 \vpran{\frac{|v - \Vout|}{\beta}} \dv} \abs{\vint{f_2}(y) + \vint{f_3}(y)} \dy
\\
& \hspace{-2.5cm}
\leq
  c_0  \underbrace{\vpran{\int_\Omega 
   \vpran{\int_{\Ss^1} \frac{1}{\theta \beta}
   \psi_0 \vpran{\frac{|y + \tau_+(y, v) v - \Xout|}{\theta}}
   \psi_0 \vpran{\frac{|v - \Vout|}{\beta}} \dv }^{p'_2}\dy}^{\frac{1}{p'_2}}}_T
   \norm{\vint{f_2}}_{L^{p_2}(\Omega)} \,.
\end{align*}
where$\frac{1}{p'_2} + \frac{1}{p_2} = 1$ and the last step follows from H\"{o}lder inequality and~\eqref{bound:vint-f-3}. The factor $T$ is estimated as follows.
\begin{align*}
  T^{p'_2}
&= \int_\Omega 
   \vpran{\int_{\Ss^1} \frac{1}{\theta \beta}
   \psi_0 \vpran{\frac{|y + \tau_+(y, v) v - \Xout|}{\theta}}
   \psi_0 \vpran{\frac{|v - \Vout|}{\beta}} \dv }^{p'_2}\dy
\\
& \hspace{-0.2cm} 
\leq
  \vpran{\int_{\Ss^1} 
  \vpran{\int_\Omega \frac{1}{\theta^{p'_2}}
  \psi^{p'_2}_0 \vpran{\frac{|y + \tau_+(y, v) v - \Xout|}{\theta}} \dy}
  \frac{1}{\beta} \psi_0 \vpran{\frac{|v - \Vout|}{\beta}} \dv}
  \vpran{\int_{\Ss^1}  \frac{1}{\beta} \psi_0 \vpran{\frac{|v - \Vout|}{\beta}} \dv}^{\frac{p'_2}{p_2}}
\\
& \hspace{-0.2cm} 
\leq
  c_0 \int_{\Ss^1} 
  \underbrace{\vpran{\int_\Omega \frac{1}{\theta^{p'_2}}
  \psi^{p'_2}_0 \vpran{\frac{|y + \tau_+(y, v) v - \Xout|}{\theta}} \dy}}_{T_1}
  \frac{1}{\beta} \psi_0 \vpran{\frac{|v - \Vout|}{\beta}} \dv \,.
\end{align*}
For each $v \in \Ss^1$, if we apply the change of variables 
\begin{align*}
   x = y + \tau_+(y, v) v \in \del\Omega^+_v \,,
\end{align*} 
with $\del\Omega^+_v$ defined in~\eqref{def:Omega-v-plus}, then $T_1$ satisfies
\begin{align*}
  T_1 
= \int_{\del\Omega^+_v} \int_0^{\tau_-(x, v)}
  \frac{1}{\theta^{p'_2}}
  \psi^{p'_2}_0 \vpran{\frac{|x - \Xout|}{\theta}} \ds\dx
\leq
  \vpran{\text{diam}(\Omega)}
  \int_{\del\Omega^+_v} 
  \frac{1}{\theta^{p'_2}}
  \psi^{p'_2}_0 \vpran{\frac{|x - \Xout|}{\theta}} \dx
\leq
  c_0 \theta^{-\vpran{p'_2 - 1}} \,.
\end{align*}
Therefore, $T$ is uniformly bounded in $v$ with the bound
\begin{align*}
  T^{p'_2}
\leq
  c_0 \theta^{-(p'_2 - 1)}
  \int_{\Ss^1} \frac{1}{\beta} \psi_0 \vpran{\frac{|v - \Vout|}{\beta}} \dv
\leq
  c_0 \theta^{-(p'_2 - 1)} \,.
\end{align*}
Inserting the estimate for $T$ back into $\CalM_\psi(f_3)$ and using ~\eqref{bound:vint-f-2}-\eqref{bound:vint-f-3},  we have
\begin{align*}
 \abs{\CalM_\psi(f_3)} =  \abs{ \iint_{\Gamma_+} \psi(x, v) f_3(x, v) \dGamma}
\leq
  c_0 \theta^{-\frac{p'_2 - 1}{p'_2}}
  \Eps^{-\frac{r-1}{r}} \delta^{-\frac{r-1}{r}} \eta^{\frac{1}{r}} 
= c_0 \theta^{-\frac{1}{p_2}}
   \Eps^{-\frac{r-1}{r}} \delta^{-\frac{r-1}{r}} \eta^{\frac{1}{r}} \,.
\end{align*}
We will choose the parameter properly later to make $\CalM_\psi(f_3)$ a negligible term, namely, we will choose parameters so that
\begin{equation}\label{eqn:parameter1}
\theta^{-\frac{1}{p_2}}\Eps^{-\frac{r-1}{r}} \delta^{-\frac{r-1}{r}} \eta^{\frac{1}{r}}\ll 1
\end{equation}

\medskip
\Ni \underline{\bf Contribution from $f_1$} \, We show in this part that by properly choosing the parameters, the contribution from $f_1$ is zero. The formula under consideration is
\begin{align*}
   \CalM_\psi(f_1) = \iint_{\Gamma_+} \psi(x, v) f_1(x, v) \dGamma \,,
\end{align*}
where we solve equation~\eqref{eq:scattering-1} to obtain
\begin{align*}
   f_1(x, v) = e^{-\int_0^{\tau_-(x, v)} \sigma_a(x - sv) \ds}
      \phi(x - \tau_-(x, v) v, v) \,,
\qquad
  (x, v) \in \bar\Omega \times \Ss^1 \,.
\end{align*}
Definitions of $\psi$ and $\phi$ give
\begin{align*}
  \CalM_\psi (f_1)
&= \frac{1}{\Eps \delta} \frac{1}{\theta \beta}
   \iint_{\Gamma_+} e^{-\int_0^{\tau_-(x, v)} \sigma_a(x - sv) \ds}
   \psi_0 \vpran{\frac{|x - \Xout|}{\theta}}
   \psi_0 \vpran{\frac{|v - \Vout|}{\beta}}
\\
& \hspace{2.5cm} \times 
   \phi_0 \vpran{\frac{(x - \tau_-(x, v) v - \Xin) \cdot \Vin_\perp}{\Eps \eta}}
    \phi_0 \vpran{\frac{|v - \Vin|}{\delta}} \dGamma \,.
\end{align*}
The sufficient condition for $\CalM_\psi (f_1)$ to vanish is 
\begin{align} \label{cond:vanish-M-1-supp}
   \Supp \vpran{\psi_0 \vpran{\frac{|v - \Vout|}{\beta}}}
\cap
   \Supp \vpran{\phi_0 \vpran{\frac{|v - \Vin|}{\delta}}} = \emptyset \,.
\end{align}
One sufficient condition for~\eqref{cond:vanish-M-1-supp} to hold is
\begin{align} \label{cond:dist-V-in-out-1}
   \abs{\Vout - \Vin} > \beta + \delta \,,
\end{align}
since then there does not exist any $v$ satisfying that
\begin{align*}
   \abs{v - \Vout} \leq \beta 
\qquad \text{and} \qquad 
   \abs{v - \Vin} \leq \delta \,.
\end{align*}
Recall that $\eta$ is defined in~\eqref{def:eta} as
\begin{align*}
   \eta = \Vout \cdot \Vin_\perp > 0 \,.
\end{align*}
Therefore, by~\eqref{cond:V-in-out}, we have
\begin{align*}
   \Vout \cdot \Vin = \sqrt{1 - \eta^2} \,.
\end{align*}
This gives
\begin{align*}
   \abs{\Vout - \Vin}^2 
= 2 - 2 \Vout \cdot \Vin 
= 2 - 2 \sqrt{1 - \eta^2} \,.
\end{align*}
Hence we have the estimate
\begin{align} \label{diff:Vout-Vin}
   \eta 
\leq \abs{\Vout - \Vin}
\leq  2\eta \,.
\end{align}
It is then clear that a sufficient condition for~\eqref{cond:dist-V-in-out-1} (and thus~\eqref{cond:vanish-M-1-supp}) to hold is
\begin{align} \label{cond:eta-1}
   \eta > \beta + \delta \,.
\end{align}
Such condition gives that $\CalM_\psi(f_1) = 0$.

\medskip
\Ni \underline{\bf Contribution from $f_2$}\,  The main contribution to the measurement comes from $f_2$, which we compute below. Denote such contribution as $\CalM_\psi(f_2)$. Then for any $(x, v) \in \bar\Omega \times \Ss^1$, we have
\begin{align*}
   \CalM_\psi(f_2)
&= \iint_{\Gamma_+}\int_0^{\tau_-(x, v)} \psi(x, v) e^{-\int_0^s \sigma_a(x - \tau v) \dtau}
    \sigma_s(x - sv) \vint{f_1}(x - sv)\ds \dGamma
\\
&= \iint_{\Gamma_+}\int_0^{\tau_-(x, v)} \int_{\Ss^1} \psi(x, v) e^{-\int_0^s \sigma_a(x - \tau v) \dtau}
    \sigma_s(x - sv) f_1(x - sv, w) \dw \ds \dGamma
\\
& = \iint_{\Gamma_+}\int_0^{\tau_-(x, v)} \int_{\Ss^1} \psi(x, v) e^{-\int_0^s \sigma_a(x - \tau v) \dtau}
      e^{-\int_0^{\tau_-(x-sv, w)} \sigma_a(x - sv - \tau w) \dtau}
\\
& \hspace{5cm} \times
    \sigma_s(x - sv) \phi((x-sv)'_w, w) \dw \ds \dGamma \,,
\end{align*}
where $(x - sv)'_w$ is the entry point of $x-sv$ along the direction $w$.
To simplify the notation, we denote
\begin{align*}
   H(s, x, v, w)
= e^{-\int_0^s \sigma_a(x - \tau v) \dtau}
   e^{-\int_0^{\tau_-(x-sv, w)} \sigma_a(x - sv - \tau w) \dtau}
   \sigma_s(x - sv) \,.
\end{align*}
Separate $\CalM_\psi(f_2)$ into two parts as
\begin{align*}
   \CalM_\psi(f_2)
&= \iint_{\Gamma_+}\int_0^{\tau_-(x, v)} \!\! \int_{\Ss^1} \psi(x, v) 
      H(s_0, \Xout, \Vout, \Vin) \phi((x-sv)'_w, w) \dw \ds \dGamma
\\
& \quad \, 
   + \iint_{\Gamma_+}\int_0^{\tau_-(x, v)} \!\! \int_{\Ss^1} \psi(x, v) 
      \vpran{H(s, x, v, w)
      - H(s_0, \Xout, \Vout, \Vin)}
       \phi((x-sv)'_w, w)
        \dw \ds \dGamma
\\
&\Denote \CalM_{2, 1} + \CalM_{2,2} \,.
\end{align*}
To treat the first term $\CalM_{2, 1}$ we insert the definitions of $\phi, \psi$ into $\CalM_{2,1}$ and obtain
\begin{align*}
  \CalM_{2, 1}
&= \frac{H(s_0, \Xout, \Vout, \Vin)}{\Eps \delta} \frac{1}{\theta \beta}
   \iint_{\Gamma_+}\int_0^{\tau_-(x, v)} \!\! \int_{\Ss^1}
      \psi_0 \vpran{\frac{|x - \Xout|}{\theta}}
   \psi_0 \vpran{\frac{|v - \Vout|}{\beta}}
   \phi_0 \vpran{\frac{|w - \Vin|}{\delta}}
\\
& \hspace{7cm} \times
   \phi_0 \vpran{\frac{((x-sv)'_w - \Xin) \cdot \Vin_\perp}{\Eps \eta}}
     \dw \ds \dGamma \,.
\end{align*}

Now we reformulate the second $\phi_0$-term, whose argument satisfies
\begin{align} \label{reform:x-sv-w}
  (x - sv)'_w
&= (x - sv) - \tau_-(x -  sv, w) w  \nn
\\
&= (\Xout - s\Vout) - \tau_-(x -  sv, w) \Vin
     + R(x, v, s, w) \,,
\end{align}
where the remainder term $R$ is
\begin{align*}
   R(x, v, s, w)
= (x - \Xout) - s(v - \Vout) - \tau_-(x -  sv, w) \vpran{w - \Vin}   \,.
\end{align*}
By Corollary~\ref{cor:1}, we have that $\nabla_x \tau_-(\cdot, w)$ is uniformly bounded in $w$ if we choose
\begin{align*}
   \theta + \delta + \beta < \gamma_\ast \,.
\end{align*}
Then by using~\eqref{def:x-0} again, we have
\begin{align*}
   \frac{((x-sv)'_w - \Xin) \cdot \Vin_\perp}{\Eps \eta}
= \frac{\vpran{\Xout - s\Vout - \Xin + R} \cdot \Vin_\perp}{\Eps \eta}
= \frac{s_0 - s + \frac{1}{\eta} R \cdot \Vin_\perp}{\Eps} \,.
\end{align*}
Let $z$ be the new variable given by
\begin{align*}
   z = s - \frac{1}{\eta} R \cdot \Vin_\perp \,.
\end{align*}
Then 
\begin{align*}
   \frac{\del z}{\del s}
= 1 - \frac{1}{\eta} \frac{\del R}{\del s} \cdot \Vin_\perp
= 1 + \frac{1}{\eta} (v - \Vout) \cdot \Vin_\perp
   - \frac{1}{\eta} \big(v \cdot \nabla_x \tau_-(x - sv, w) \big) (w - \Vin) \cdot \Vin_\perp \,.
\end{align*}
Due to the compact supports of $\phi_0$ and $\psi_0$, the variables $(x, v, w)$ in $R$ satisfy that
\begin{align*}
   \abs{x - \Xout} \leq \theta \,,
\qquad
   \abs{v - \Vout} \leq \beta \,,
\qquad
   \abs{w- \Vin} \leq \delta \,.
\end{align*}
If we impose that
\begin{align} \label{cond:eta-2}
   \eta \gg \beta + \delta \,,
\end{align}
then $\del z/\del s > 1/2$ and we can make the change of variable from $s$ to $z$. Denote $I = z^{-1}(0, \tau_-(x, v))$. Then
\begin{align*}
 \lim_{\Eps, \theta \to 0} \lim_{\stackrel{\eta \to 0}{\eta \gg \beta + \delta}} \CalM_{2,1}
&=  \lim_{\Eps, \theta \to 0} \lim_{\stackrel{\eta \to 0}{\eta \gg \beta + \delta}} \frac{H(s_0, \Xout, \Vout, \Vin)}{\Eps \delta} \frac{1}{\theta \beta}
   \iint_{\Gamma_+}\int_{I} \!\! \int_{\Ss^1}
      \psi_0 \vpran{\frac{|x - \Xout|}{\theta}}
   \psi_0 \vpran{\frac{|v - \Vout|}{\beta}}
\\
& \hspace{6cm} \times
   \phi_0 \vpran{\frac{s_0 - z}{\Eps}} 
   \phi_0 \vpran{\frac{|w - \Vin|}{\delta}} \frac{\del z}{\del s}
     \dw \ds \dGamma  \,.
\end{align*}
Since $s_0$ is an interior point by Corollary~\ref{cor:2}, we have
\begin{align*}
   \lim_{\Eps, \theta \to 0} \lim_{\stackrel{\eta \to 0}{\eta \gg \beta + \delta}} \CalM_{2,1}
= H(s_0, \Xout_0, \Vin, \Vin)  \,.
\end{align*}
where $\Xout, \Vout$ are replaced by $\Xout_0, \Vin$ in the limit $\eta \to 0$. Meanwhile, by the continuity of $\tau_-$ and $\sigma_a$, the second term $\CalM_{2,2}$ will vanish in the limit.

Consider that under conditions~\eqref{eqn:parameter1} and~\eqref{cond:eta-1}, assuming $\delta^{-\frac{r-1}{r}} \eta^{\frac{1}{r}}\to 0$, then $\CalM_\psi(f_1) = 0$ and $\CalM_\psi(f_3)\to0$, overall we have
\begin{align*}
  \sigma_s(x_0)
&= e^{\int_0^{s_0} \sigma_a(\Xout - \tau \Vout) \dtau}
   e^{\int_0^{\tau_-(x_0, \Vin)} \sigma_a(x_0 - \tau \Vin) \dtau}  
   \lim_{\Eps, \theta \to 0} \lim_{\stackrel{\eta \to 0}{\eta \gg \beta + \delta}} \CalM_\psi(f)
\\
&= e^{\int_0^{s_0} \sigma_a(\Xout_0 - \tau \Vin) \dtau}
   e^{\int_0^{\tau_-(x_0, \Vin)} \sigma_a(x_0 - \tau \Vin) \dtau}
      \lim_{\Eps, \theta \to 0} \lim_{\stackrel{\eta \to 0}{\eta \gg \beta + \delta}} \CalM_\psi(f) \,.
\end{align*}

\medskip
\Ni \underline{\bf Choice of the parameters}\, We now collect all requirements on the parameters, namely equation~\eqref{eqn:parameter1}, \eqref{cond:eta-1} and~\eqref{cond:eta-2}. Choose $\theta \to 0$ and $\Eps \to 0$ independent of $\eta$, these requirements reduce to:
\begin{align} \label{cond:param-1}
   \delta^{-\frac{r-1}{r}} \eta^{\frac{1}{r}}
\ll 1 \,,
\qquad
   \beta + \delta \ll \eta \,.
\end{align}
In the borderline case where $\delta = \eta$, the sufficient condition for the first inequality in ~\eqref{cond:param-1} to hold is 
\begin{align*}
  \frac{r-1}{r} < \frac{1}{r} 
\Rightarrow
  r < 2 \,.
\end{align*}
This suggests that we can find proper parameters by letting $\theta \to 0$ and $\Eps \to 0$ independent of $\eta$ and setting
\begin{align*}
   \beta = \delta = \eta^{1 + \beta_0} \,.
\end{align*}
with $\beta_0$ small enough, then~\eqref{cond:param-1} holds for $r \in (1, 2)$.
\end{proof}

\appendix

\section{Well-posedness of the Nonlinear RTE} \label{sec:appendix-1}
In this appendix we use the classical monotonicity method combined with the Schauder fixed-point argument to show that the nonlinear RTE given in~\eqref{eq:RTE-nonlinear-3}-\eqref{eq:RTE-nonlinear-4} is well-posed. Recall that the equations are given by
\begin{align}
   v \cdot \nabla_x I &= -\sigma_a I + \sigma_a T^4 \,, 
\qquad
   I \big|_{\Gamma_-} = \phi(x, v) \,,
   \label{eq:RTE-nonlinear-5}
\\
   \Delta_x T &= \sigma_a T^4 - \sigma_a \vint{I},
\qquad
   T \big|_{\del\Omega} = 0 \,.
   \label{eq:RTE-nonlinear-6}
\end{align} 
where $\phi \geq 0$ and $\phi \in L^\infty(\Gamma_-)$. The statement of the well-posedness result is
\begin{thm} \label{thm:well-posed-nonlinear-RTE}
Suppose $\phi \in L^\infty(\Gamma_-)$ and $\phi \geq 0$. Then~\eqref{eq:RTE-nonlinear-5}-\eqref{eq:RTE-nonlinear-6} has a unique solution. 
\end{thm}
\begin{proof}
Let $\CalD$ be the solution set given by
\begin{align*}
   \CalD = \{T \big| \,\, 0 \leq T \leq \norm{\phi}_{L^\infty} \,\} \,.
\end{align*}
Take $H \in \CalD$. We want to construct a map $\CalF$ and show that $\CalF(H) \in \CalD$. Let $I_H$ be the solution such that
\begin{align*}
     v \cdot \nabla_x I_H &= -\sigma_a I_H + \sigma_a H^4 \,, 
\qquad
   I_H \big|_{\Gamma_-} = \phi(x, v) \,.
\end{align*}
Such $I_H$ exists by a direct integration along the characteristics. Since $H^4 \geq 0$ and $\phi \geq 0$, we have $I_H \geq 0$. Moreover, if we consider $I_H - \norm{\phi}_{L^\infty}$, then it satisfies
\begin{align*}
    v \cdot \nabla_x \vpran{I_H - \norm{\phi}_{L^\infty}}
&= -\sigma_a \vpran{I_H - \norm{\phi}_{L^\infty}} 
     + \sigma_a \vpran{H^4 - \norm{\phi}_{L^\infty}}\,, 
\qquad
   \vpran{I_H - \norm{\phi}_{L^\infty}} \big|_{\Gamma_-} \leq 0 \,.
\end{align*}
Since $H^4 - \norm{\phi}_{L^\infty} \leq 0$, we have $I_H \leq \norm{\phi}_{L^\infty}$. Define $\CalF(H) = T$ where $T$ is the solution to the equation
\begin{align*}
   \Delta_x T &= \sigma_a T^4 - \sigma_a \vint{I_H},
\qquad
   T \big|_{\del\Omega} = 0 \,.
\end{align*}
or equivalently,
\begin{align}
  - \Delta_x T &= -\sigma_a T^4 + \sigma_a \vint{I_H},
\qquad
   T \big|_{\del\Omega} = 0 \,.
   \label{eq:RTE-nonlinear-7}
\end{align}
We use the classical monotonicity method for semilinear elliptic equations to show that such $T$ exists and is unique. First, let $\underline{T} = 0$ and $\bar T$ be the unique solution to the equation
\begin{align*}
     -\Delta_x \bar T &=  \sigma_a \vint{I_H} \,,
\qquad
   \bar T |_{\del\Omega} = 0 \,.
\end{align*}
Since it holds that
\begin{align*}
   -\Delta_x \underline{T} - \sigma_a \vint{I_H}
\leq 0 =  -\sigma_a \underline{T}^4 ,
\qquad
   \underline{T} \big|_{\del\Omega} = 0 \,,
\end{align*}
and
\begin{align*}
   - \Delta_x \bar{T} - \sigma_a \vint{I_H}
= 0
\geq -\sigma_a \bar {T}^4 ,
\qquad
   \bar T \big|_{\del\Omega} = 0 \,,
\end{align*}
the functions $\underline{T}$ and $\bar T$ serve as the sub- and super-solutions of~\eqref{eq:RTE-nonlinear-7}. Moreover, we have
\begin{align*}
   0 \leq \underline{T} \leq \bar T \,.
\end{align*}
We use an inductive argument to build an increasing sequence as follows. Fix a constant $\lambda$ which satisfies 
\begin{align*}
   \lambda > 4 \norm{\sigma_a}_{L^\infty} \norm{\phi}_{L^\infty}^{3/4} \,.
\end{align*} 
This guarantees that the function $f(x) = \lambda x - \sigma_a x^4$ is increasing for any $x \in (0, \,  \norm{\phi}_{L^\infty})$. Initialize the sequence at $T_0 = \underline{T}$ and suppose at the inductive step that 
\begin{align*}
   0 \leq T_k \leq \norm{\phi}_{L^\infty}^{1/4} \,.
\end{align*} 
Define $T_{k+1}$ as the unique solution to the equation
\begin{align}
  - \Delta_x T_{k+1} + \lambda T_{k+1}
&= \lambda T_k-\sigma_a T_k^4 
     + \sigma_a \vint{I_H},
\qquad
   T_{k+1} \big|_{\del\Omega} = 0 \,.
   \label{eq:RTE-nonlinear-8}
\end{align}
Note that $T_{k+1} \geq 0$ since by the choice of $\lambda$ and the assumption of $T_k$ the right-hand side satisfies
\begin{align*}
  \lambda T_k-\sigma_a T_k^4 
     + \sigma_a \vint{I_H}
\geq 
   \sigma_a \vint{I_H}
\geq 0 \,.
\end{align*}
Moreover, $T_{k+1} \leq \norm{\phi}_{L^\infty}^{1/4}$ since we have
\begin{align*}
     - \Delta_x T_{k+1} + \lambda T_{k+1}
\leq \lambda T_k \,,
\end{align*}
which implies that
\begin{align*}
   \max T_{k+1} \leq \max T_k \leq \norm{\phi}_{L^\infty}^{1/4} \,.
\end{align*}
Now we show that $T_{k+1} \geq T_k$ for all $k \geq 0$. First, $T_1 \geq T_0 = 0$ since we have shown that $T_k \geq 0$ for all $k$.
Next, the difference $T_{k+1} - T_k$ satisfies the equation
\begin{align*}
  - \Delta_x \vpran{T_{k+1} - T_k} 
  + \lambda \vpran{T_{k+1} - T_k}
&= f(T_k) - f(T_{k-1}) \geq 0,
\qquad
   (T_{k+1} - T_k) \big|_{\del\Omega} = 0 \,.  
\end{align*}
where recall that $f(x) = \lambda x - x^4$. Hence
\begin{align*}
   \min_{\bar\Omega} (T_{k+1} - T_k) 
= \min_{\del\Omega} (T_{k+1} - T_k) 
= 0 \,,
\end{align*}
which implies that $T_{k+1} \geq T_k$. We thereby have constructed an increasing sequence. Lastly we want to show that $T_k \leq \bar T$ for all $k \geq 0$. This is done by considering the equation for $T_k - \bar T$ which reads
\begin{align*}
  - \Delta_x \vpran{T_k - \bar T} 
  + \lambda \vpran{T_k - \bar T}
&= f(T_{k-1}) -  \lambda \bar T,
\qquad
   (T_k - \bar T) \big|_{\del\Omega} = 0 \,.  
\end{align*}
By the induction assumption at $k$ such that $T_{k-1} \leq \bar T$, the right-hand side of the equation satisfies
\begin{align*}
   f(T_{k-1}) - \lambda \bar T
\leq
  f(T_{k-1}) - \vpran{\lambda \bar T - \sigma_a \bar T^4}
\leq 0 \,.
\end{align*}
Hence by the maximum principle, we have
\begin{align*}
   \max_{\bar\Omega} (T_{k+1} - \bar T) 
= \max_{\del\Omega} (T_{k+1} - \bar T) 
= 0 \,,
\end{align*}
which gives that $T_{k+1} \leq \bar T$. Overall, we have
\begin{align*}
   0 = \underline{T} = T_0 
\leq T_1 
\leq \cdots \leq T_k \leq \cdots \leq \bar T \,.
\end{align*}
Together with the $L^\infty$ bound of $T_k$, we have that there exists $T \in L^\infty(\Omega)$ such that
\begin{align*}
   T_k \to T 
\quad
\text{pointwise and in $L^4$} \,.
\end{align*}
Passing $k \to \infty$ in~\eqref{eq:RTE-nonlinear-8} shows $T$ is a weak solution of~\eqref{eq:RTE-nonlinear-7} and $\norm{T}_{L^\infty} \leq \norm{\phi}_{L^\infty}^{1/4}$. The $L^\infty$-bounds of $T$ and $I_H$ shows that $T \in W^{2, \infty}(\Omega)$. Hence the mapping $\CalF$ is compact and we can then apply the Schauder fixed-point theorem to obtain a strong solution to~\eqref{eq:RTE-nonlinear-5}-\eqref{eq:RTE-nonlinear-6}. The uniqueness can be shown by directly taking the difference of two potential solutions and using the energy estimate. 
\end{proof}

\section{Geometry}\label{sec:geometry}
In this appendix, we show the proofs for two geometric relations~\eqref{geometry-1} and~\eqref{geometry-2}. First we prove~\eqref{geometry-1}. 

\begin{proof}[Proof of~\eqref{geometry-1}]
Suppose that in a small neighborhood of $x \in \del\Omega$, the boundary $\del\Omega$ is parametrized as 
\begin{align*}
   x = x(u) \,,
\qquad
   u \in (u_0, u_1) \,.
\end{align*}
Then the corresponding small neighborhood of $y$, given that $y$ is the exit point of $x$, is also parametrized by $u$ through the relation
\begin{align*}
    y = y(u) = x(u) - \tau_-(x(u), \Vin) \Vin \,,
\qquad
   u \in (u_0, u_1) \,.
\end{align*}
Therefore, $\frac{\dx}{\du}$ and $\frac{\dy}{\du}$ are both along the tangential direction. Moreover, 
\begin{align*}
   \dS_x = \abs{\frac{\dx}{\du}} \du \,,
\qquad
  \dS_y = \abs{\frac{\dy}{\du}} \du \,.
\end{align*}
which gives
\begin{align*}
  \frac{\dS_x}{\dS_y} = \frac{\abs{\dx/\du}}{\abs{dy/\du}} \,.
\end{align*}
Note that for any unit vectors $\alpha, \beta$, we have
\begin{align} \label{perp}
   \abs{\alpha \cdot \beta^\perp} = \abs{\alpha^\perp \cdot \beta} \,. 
\end{align}
Therefore, if we denote $T_x$ and $T_y$ as the unit tangential directions at $x$ and $y$ respectively, then 
\begin{align*}
   \abs{n(x) \cdot \Vin}
= \abs{T_x^\perp \cdot \Vin}
= \abs{T_x \cdot \vpran{\Vin}^\perp}
= \abs{\frac{\dx}{\du} \cdot \vpran{\Vin}^\perp}
   \frac{1}{\abs{\frac{\dx}{\du}}}
\end{align*}
Similarly, 
\begin{align*}
   \abs{n(y) \cdot \Vin}
= \abs{\frac{\dy}{\du} \cdot \vpran{\Vin}^\perp}
   \frac{1}{\abs{\frac{\dy}{\du}}} \,.
\end{align*}
Therefore, 
\begin{align*}
    \frac{\abs{n(y) \cdot \Vin}}{\abs{n(x) \cdot \Vin}}
= \frac{\abs{\dx/\du}}{\abs{dy/\du}}
   \frac{\abs{\frac{\dy}{\du} \cdot \vpran{\Vin}^\perp}}{\abs{\frac{\dx}{\du} \cdot \vpran{\Vin}^\perp}} \,.
\end{align*}
Observe that by the definition of $y$, we have
\begin{align*}
   \frac{\dy}{\du}
= \frac{\dx}{\du} - \frac{\dtau_-(x(u), \Vin)}{\du} \Vin \,.
\end{align*}
Hence,
\begin{align*}
   \frac{\dy}{\du} \cdot \vpran{\Vin}^\perp
= \vpran{\frac{\dx}{\du} - \frac{\dtau_-(x(u), \Vin)}{\du} \Vin} \cdot \vpran{\Vin}^\perp
= \frac{\dx}{\du} \cdot \vpran{\Vin}^\perp \,.
\end{align*}
Therefore, 
\begin{align*}
   \frac{\abs{n(y) \cdot \Vin}}{\abs{n(x) \cdot \Vin}}
= \frac{\abs{\dx/\du}}{\abs{dy/\du}}
= \frac{\dS_x}{\dS_y} \,,
\end{align*}
which is equivalent to~\eqref{geometry-1}. 
\end{proof}

Next we verify~\eqref{geometry-2}.
\begin{proof}[Proof of~\eqref{geometry-2}]
Fix $x \in \del\Omega_v^+$. Suppose the neighborhood of $x$ (in $\del\Omega$) is a curve parametrized as 
\begin{align*}
   x = x(u) \,,
\qquad
  u \in (u_0, u_1) \,.
\end{align*}
Then 
\begin{align*}
   y(u, s) = x(u) - s v \,.
\end{align*}
The Jacobian of the mapping $y \to (u, s)$ is
\begin{align*}
   \abs{\det 
   \begin{pmatrix}
      \frac{\dy_1}{\du}  & \frac{\dy_2}{\du}  \\[2pt]
      \frac{\dy_1}{\ds}  & \frac{\dy_2}{\ds}
   \end{pmatrix}}
= \abs{\det 
   \begin{pmatrix}
      \frac{\dx_1}{\du}  & \frac{\dx_2}{\du}  \\[2pt]
      -v_1  & -v_2
   \end{pmatrix}}
= \abs{-v_2 \frac{\dx_1}{\du} + v_1 \frac{\dx_2}{\du}}
= \abs{v^\perp \cdot \nabla_u x}
= \abs{v^\perp \cdot T_x} \abs{\nabla_u x}  \,.
\end{align*}
where $T_x$ is the tangent direction at $x$. By~\eqref{perp}, we have
\begin{align*}
   \abs{\det 
   \begin{pmatrix}
      \frac{\dy_1}{\du}  & \frac{\dy_2}{\du}  \\[2pt]
      \frac{\dy_1}{\ds}  & \frac{\dy_2}{\ds}
   \end{pmatrix}}
= \abs{v \cdot n(x)} \abs{\nabla_u x} \,.
\end{align*}
Therefore~\eqref{geometry-2} holds since
\begin{align*}
   \dy 
= \abs{\frac{\del(y_1, y_2)}{\del (u, s)}} \du\ds
= \abs{v \cdot n(x)} \abs{\nabla_u x} \du\ds
= n(x) \cdot v \dS_x \ds \,,
\end{align*}
where we can remove the absolute value sign since $n(x) \cdot v > 0$.
\end{proof}

\section{Some technical lemmas}
This appendix is devoted to showing several technical results used in the proof of Theorem~\ref{thm:recover-scattering}. The notations $\Xin, \Xout_0, \Vin, s_0, x_0, \Vout$ represent the same quantities as in the theorem.

\begin{lem} \label{lem:C-1-tau-minus}
There exists $\gamma_0$ small enough such that $\tau_-(x-sv, w)$ is $C^1$ in $(x, v, s, w)$ over the domain
\begin{align} \label{cond:neighbourhood}
   \abs{x - \Xout_0} + \abs{w - \Vin} + \abs{v - \Vin} < \gamma_0 \,,
\qquad
   s \in (0, \tau_-(x, v)) \,,
\qquad
  (x, v) \in \Gamma_+ \,.
\end{align}
Moreover, the bound $\norm{\nabla_x \tau_-(\cdot, w)}_{L^\infty}$ is independent of $w$ over the region~\eqref{cond:neighbourhood}.
\end{lem}
\begin{proof}
By Lemma~\ref{lem:tau-continuity}, we only need to show that there exists a constant $c_{0,1} > 0$ such that
\begin{align} \label{cond:non-deg-1}
   w \cdot n((x - sv)_-) < - c_{0, 1} < 0
\end{align}
for any $(x, v, s, w)$ satisying~\eqref{cond:neighbourhood}, recalling that $(x - sv)_-$ is the backward exit point of $x - sv$. The idea is to show that $(x - sv)_-$ is close to $\Xin$ when $\gamma_0$ is small. Then by the continuity of the outward normal $n$, we obtain~\eqref{cond:non-deg-1} from the non-degeneracy condition at $(\Xin, \Vin)$. The closeness of $(x - sv)_-$ to $\Xin$ is fairly evident from the geometry shown in Figure~\ref{Fig:2}. 
\begin{figure}[h] 
\includegraphics[scale = 0.36]{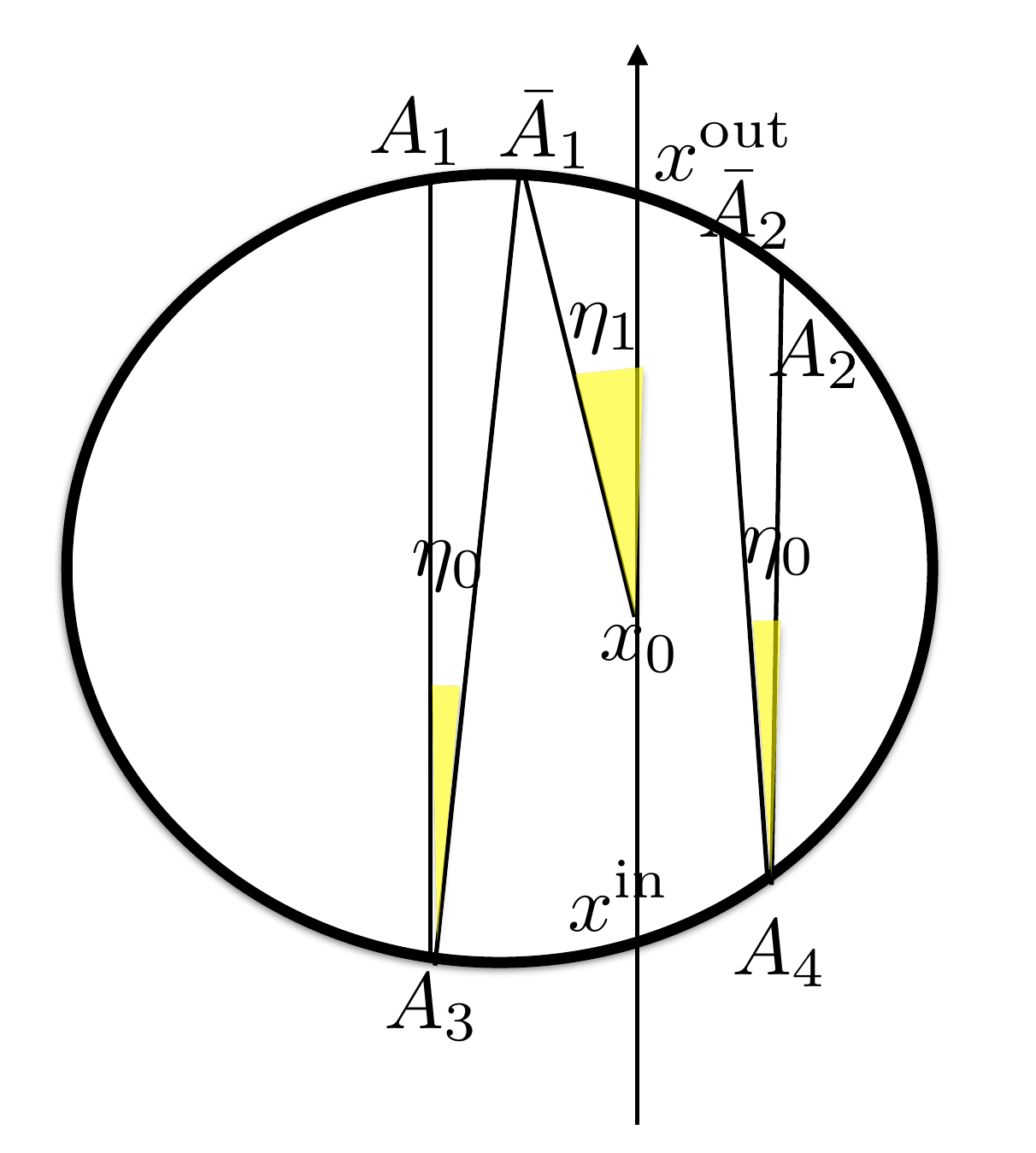}
\caption{Geometry for non-degeneracy}\label{Fig:2}
\end{figure}

For a rigorous proof, we first assume, via a proper rotation and translation, that $\Vin$ is along the positive $y$-axis and $\Xin$ and $\Xout_0$ are both on the $y$-axis. Since $\Omega$ is convex and $\Vin \cdot n(\Xin) \neq 0$, we have
\begin{align*}
   \Vin \cdot n(\Xout_0) > 0 \,.
\end{align*}Take small neighborhoods $\mathcal{N}(\Xin), \mathcal{N}(\Xout_0)$ around $\Xin$ and $\Xout_0$ on $\del\Omega$ such that
\begin{align*}
   \Vin \cdot n(x) &< \frac{1}{2} \Vin \cdot n(\Xin) < 0 \,,
\qquad \,\,
  \forall \, x \in \mathcal{N}(\Xin) \,,
\\
   \Vin \cdot n(x) &> \frac{1}{2} \Vin \cdot n(\Xout_0) > 0 \,,
\qquad
  \forall \, x \in \mathcal{N}(\Xout_0) \,,
\end{align*}
Denote the boundary vertices of $\mathcal{N}(\Xin), \mathcal{N}(\Xout_0)$ as $A_1, A_2, A_3, A_4$. By adjusting the sizes of $\mathcal{N}(\Xin), \mathcal{N}(\Xout_0)$ we can choose these vertices in the way such that
\begin{align*}
   A_1 A_3 \Parallel A_2 A_4 \Parallel y-\text{axis} \,.
\end{align*}
Choose $\bar {A}_1$ and $\bar A_2$ as two points on $\text{arc}(A_1 \Xout_0)$ and $\text{arc}(A_2 \Xout_0)$ respectively such that 
\begin{align*}
   \angle A_1 A_3 \bar A_1
= \angle A_2 A_4 \bar A_2
= : \eta_0 \,.
\end{align*}
Denote the region bounded by the line segments $\bar A_1 A_3$, $\bar A_2 A_4$ and the two arcs $\text{arc}(\bar A_1 \bar A_2)$, $\text{arc}(A_3 A_4)$ as $D_0$. Then for any $(x, v) \in \Gamma_+$ with $\cos^{-1}(v \cdot \Vin) < \eta_0$ and any $s \in (0, \tau_-(x, v))$, we have
\begin{align*}
   (x - sv)_- \in \mathcal{N}(\Xin) \,.
\end{align*}
Hence, for such $(x, v, s)$ we have
\begin{align*}
   \Vin \cdot n((x - sv)_-) < \frac{1}{2} \Vin \cdot n(\Xin) < 0  \,.
\end{align*}
Take $\gamma_0$ small enough such that 
\begin{align*}
   \gamma_0 
< \min\left\{\frac{1}{2} \eta_0, \,\, \frac{1}{4} \Vin \cdot n(\Xin), \,\,
   \abs{\bar A_1 \Xout_0}, \, \abs{\bar A_2 \Xout_0}\right\} \,.
\end{align*}
Then
\begin{align*}
   w \cdot n((x - sv)_-)
< \frac{1}{4} \Vin \cdot n(\Xin) 
< 0
\end{align*}
for any $(x, v, s, w)$ satisfying~\eqref{cond:neighbourhood}. Hence $\tau_-$ is $C^1$ over the region~\eqref{cond:neighbourhood}. The explicit formula for $\nabla_x \tau_-$ in Lemma~\ref{lem:tau-continuity} shows that $\norm{\nabla_x \tau_-(\cdot, w)}_{L^\infty}$ is uniformly bounded in $w$. 
\end{proof} 

Two immediate consequences follow. 
\begin{cor} \label{cor:1}
There exist $\eta_\ast, \gamma_\ast$ such that if $\eta$ in Theorem~\ref{thm:recover-scattering} satisfies $\eta < \eta_\ast$, then $\tau_-(x-sv, w)$ is $C^1$ in $(x, v, s, w)$ over the domain
\begin{align} \label{cond:neighbourhood-1}
   \abs{x - \Xout} + \abs{w - \Vin} + \abs{v - \Vout} < \gamma_\ast \,,
\qquad
   s \in (0, \tau_-(x, v)) \,,
\qquad
  (x, v) \in \Gamma_+ \,.
\end{align}
Moreover, the bound $\norm{\nabla_x \tau_-(\cdot, w)}_{L^\infty}$ is independent of $w$ over the region~\eqref{cond:neighbourhood}.
\end{cor}
\begin{proof}
By Lemma~\ref{lem:C-1-tau-minus}, we only need to show that $\Xout$ is close to $\Xout_0$ and $\Vout$ is close to $\Vin$ by taking $\eta_\ast$ small. By~\eqref{diff:Vout-Vin}, if we taking $\eta_\ast < \frac{1}{8} \gamma_0$, then
\begin{align*}
   \abs{\Vout - \Vin} \leq 2 \eta < \frac{1}{4} \gamma_0 \,.
\end{align*}
Denote the angle $\angle \bar A_1 x_0 \Xout_0$ as $\bar \eta$. Then for $\eta_\ast < \bar \eta$, the point $\Xout$ is on arc$(\bar A_1 \Xout_0)$. Since
\begin{align*}
   \lim_{\bar \eta \to 0} \abs{\bar A_1 - \Xout_0} = 0 \,,
\end{align*}
by choosing $\eta_\ast$ small enough, we have
\begin{align*}
   \abs{\Xout - \Xout_0} < \frac{1}{4} \gamma_0 \,.
\end{align*}
Hence if we let $\gamma_\ast = \frac{1}{2} \gamma_0$, then for any $(x, v, s, w)$ in the region~\eqref{cond:neighbourhood-1}, they also satisfy that
\begin{align*}
& \quad \,
   \abs{x - \Xout_0} + \abs{w - \Vin} + \abs{v - \Vin}
\\
&\leq
  \abs{x - \Xout} + \abs{w - \Vin} + \abs{v - \Vout}
  + \abs{\Xout - \Xout_0} + \abs{\Vout - \Vin}
\\
& < \frac{1}{2}\gamma_0 + \frac{1}{4}\gamma_0 
       + \frac{1}{4}\gamma_0 = \gamma_0 \,,
\end{align*}
whereby Lemma~\ref{lem:C-1-tau-minus} applies. 
\end{proof}

\begin{cor} \label{cor:2}
Let $\gamma_\ast$ be the upper bound such that $\tau_-$ is $C^1$ in the domain~\eqref{cond:neighbourhood-1}. Then for $\gamma_\ast$ small enough, $s_0$ is always an interior point in $(0, \tau_-(x, v))$ whenever $(x, v)$ satisfies~\eqref{cond:neighbourhood-1}. 
\end{cor}
\begin{proof}
First recall that $s_0 \in (0, \tau_-(\Xout, \Vout))$. Then
\begin{align*}
   \sigma_0 := \tau_-(\Xout, \Vout) - s_0 > 0 \,.
\end{align*}
By Corollary~\ref{cor:1}, the backward exist time $\tau_-(x, v)$ is continuous for $(x, v)$ in the closure of the domain dictated by~\eqref{cond:neighbourhood-1}.  Hence if $\gamma_\ast$ is small enough, then
\begin{align*}
   \abs{\tau_-(x, v) - \tau_-(\Xout, \Vout)} < \frac{1}{2} \sigma_0 \,.
\end{align*}
Therefore, 
\begin{align*}
   \tau_-(x, v) - s_0
> \tau_-(\Xout, \Vout) - \frac{1}{2} \sigma_0 - s_0
= \sigma_0 > 0 \,,
\end{align*}
which shows $s \in (0, \tau_-(x, v))$.
\end{proof}

\bibliographystyle{amsxport}
\bibliography{transbib}

\end{document}